\documentclass[a4paper,11pt]{article}

\usepackage{verbatim}
\usepackage[T1]{fontenc}
\usepackage{lmodern}
\usepackage[latin1]{inputenc}
\usepackage{setspace}
\usepackage{indentfirst}
\usepackage{bbm}

\usepackage{geometry}
\usepackage{hyperref}

\usepackage{times}
\usepackage{amsthm}
\usepackage{amsmath}
\usepackage{amsfonts}
\usepackage{amssymb}
\usepackage{bibentry}
\usepackage{color}
\usepackage{mathrsfs}
\usepackage{breqn}
\usepackage{stmaryrd}
\SetSymbolFont{stmry}{bold}{U}{stmry}{m}{n}

\usepackage{bm}

\usepackage{tikz-cd}
\usepackage{tikz}
\usetikzlibrary{matrix,arrows,decorations.pathmorphing}
\usepackage{graphicx}
\usepackage{subfigure}
\usepackage[toc,page]{appendix}

\DeclareMathOperator{\Hom}{Hom}
\DeclareMathOperator{\End}{End}

\DeclareMathOperator{\Aut}{Aut}

\DeclareMathOperator{\Ad}{Ad}
\DeclareMathOperator{\ad}{ad}

\newtheorem{Thm}{Theorem}[section]

\newtheorem{Pro}[Thm]{Proposition}
\newtheorem{Lem}[Thm]{Lemma}
\newtheorem{Cor}[Thm]{Corollary}
\newtheorem{Def-Pro}[Thm]{Definition-Proposition}

\theoremstyle{definition}
\newtheorem{Ex}[Thm]{Example}
\newtheorem{Rm}[Thm]{Remark}

\begin{document}
\title{Coadjoint orbits of Lie groupoids}
\author{Honglei Lang \, and  \, Zhangju Liu \\\vspace{1mm}
\it{Max Planck Institute for Mathematics, Bonn, 53111, Germany}\\ 
\it{Department of Mathematics, Peking University, Beijing, 100871, China}\\
 hllang@mpim-bonn.mpg.de,~~liuzj@pku.edu.cn
}


\date{}
\maketitle

\makeatletter
\newif\if@borderstar
\def\bordermatrix{\@ifnextchar*{%
\@borderstartrue\@bordermatrix@i}{\@borderstarfalse\@bordermatrix@i*}%
}
\def\@bordermatrix@i*{\@ifnextchar[{\@bordermatrix@ii}{\@bordermatrix@ii[()]}}
\def\@bordermatrix@ii[#1]#2{%
\begingroup
\m@th\@tempdima8.75\p @\setbox\z@\vbox{%
\def\cr{\crcr\noalign{\kern 2\p@\global\let\cr\endline }}%
\ialign {$##$\hfil\kern 2\p@\kern\@tempdima & \thinspace %
\hfil $##$\hfil && \quad\hfil $##$\hfil\crcr\omit\strut %
\hfil\crcr\noalign{\kern -\baselineskip}#2\crcr\omit %
\strut\cr}}%
\setbox\tw@\vbox{\unvcopy\z@\global\setbox\@ne\lastbox}%
\setbox\tw@\hbox{\unhbox\@ne\unskip\global\setbox\@ne\lastbox}%
\setbox\tw@\hbox{%
$\kern\wd\@ne\kern -\@tempdima\left\@firstoftwo#1%
\if@borderstar\kern2pt\else\kern -\wd\@ne\fi%
\global\setbox\@ne\vbox{\box\@ne\if@borderstar\else\kern 2\p@\fi}%
\vcenter{\if@borderstar\else\kern -\ht\@ne\fi%
\unvbox\z@\kern-\if@borderstar2\fi\baselineskip}%
\if@borderstar\kern-2\@tempdima\kern2\p@\else\,\fi\right\@secondoftwo#1 $%
}\null \;\vbox{\kern\ht\@ne\box\tw@}%
\endgroup
}
\makeatother
\begin{abstract}
For a  Lie groupoid $\mathcal{G}$ with Lie algebroid $A$, we realize the symplectic leaves  of the Lie-Poisson structure on $A^*$ as orbits of the affine coadjoint action of the Lie groupoid $\mathcal{J}\mathcal{G}\ltimes T^*M$ on $A^*$, which coincide with the groupoid orbits of the symplectic groupoid $T^*\mathcal{G}$ over $A^*$.  It is also shown that there is a fiber bundle structure on each symplectic leaf. In the case of gauge groupoids, a symplectic leaf is the  universal phase space 
for a classical particle in a Yang-Mills field. \end{abstract}



\section{Introduction}
Lie algebras provide basic examples of Poisson manifolds. Namely, the dual space $\mathfrak{g}^*$ of a finite-dimensional Lie algebra $\mathfrak{g}$ admits a linear Poisson structure, called its Lie-Poisson structure. 
As elements in $\mathfrak{g}$ can be understood as linear functions on $\mathfrak{g}^*$, the Poisson bracket is given by the Lie bracket. It is known that a Poisson manifold naturally decomposes into symplectic leaves, which is referred as the symplectic foliation of the Poisson manifold. 
In particular,  symplectic leaves of the Lie-Poisson structure on $\mathfrak{g}^*$ are coadjoint orbits of $G$, the connected Lie group of $\mathfrak{g}$. 
See \cite{Camile} for  examples.  

A Lie algebroid $A$ also gives a Lie-Poisson structure on its dual bundle $A^*$. 
 It is natural to figure out the structures of its  symplectic leaves.
We would expect to realize the symplectic leaves of  the Lie-Poisson structure on $A^*$  as coadjoint orbits of Lie groupoids. Two problems come out in this case: first, for Lie algebroids and Lie groupoids, there is not a normal adjoint representation. People solve this by using the  first jet algeboids and the first jet groupoids  \cite{Lu, CSS} or by constructing a representation up to homotopy 
\cite{Lu, ACG}. We shall use the jet groupoid coadjoint acting on $A^*$ to realize the symplectic leaves. Second, for the Lie algebroid $TM$, the Poisson structure on $T^*M$ is symplectic, so it is a symplectic leaf itself.  
This example tells us that a coadjoint orbit of the  first jet groupoid on $A^*$ is not always a symplectic leaf since the linear action preserves the zero section. Actually, a symplectic 
leaf differs from a coadjoint orbit by the vertical translation given by $T^*M$. We get that a symplectic leaf is an orbit of the affine coadjoint action of $\mathcal{J}\mathcal{G}\ltimes T^*M$ on $A^*$, where $\mathcal{G}$ is the $r$-connected Lie groupoid of $A$. 

It is shown in this paper that  a symplectic leaf passing $\alpha\in A_x^*$ is a fiber bundle  over $T^*\mathcal{L}_x$ with fiber type $\mathcal{S}_{i^*\alpha}$, where $\mathcal{L}_x\subset M$ is the algebroid leaf passing $x$ and $\mathcal{S}_{i^*\alpha}$ is the coadjoint orbit passing $i^*\alpha$ in the dual of the isotropy Lie algebra $\mathrm{ker}\rho_x$. 
In particular, for the gauge algebroid $TP/G$ associated with a principal $G$-bundle $P$ over $M$, we recover the  universal phase space constructed by Weinstein in \cite{Weinstein} and the phase space constructed by Sternberg 
for a classical particle in a Yang-Mills field \cite{GS}.  Examples of symplectic leaves in $A^*$ are given. For a
 regular integrable distribution $F$ of a manifold,  a symplectic leaf is the cotangent bundle of a leaf of this distribution; For the Lie algebroid $T^*_\pi M$ for a Poisson manifold $M$,  a symplectic leaf is the tangent bundle of a symplectic leaf of $M$.


One more motivation for this work is that 
a coadjoint orbit of a Lie group is a symplectic homogeneous space, and up to coverings, every symplectic homogeneous space is a coadjoint  orbit if the first and second cohomology groups of its Lie algebra vanish;
see for example \cite{GS}. 
However, for the Lie groupoid case, although a symplectic leaf is still a homogeneous space of Lie groupoids, the action of $\mathcal{J}\mathcal{G}\ltimes T^*M$ does not preserve the symplectic structure on a leaf. 
A proper definition of a symplectic homogeneous space for Lie groupoids still needs to be fixed. 
Furthermore, except the symplectic foliation structure on a Poisson manifold, Weinstein's splitting theorem \cite{Weinstein1983} states that, locally, a poisson manifold is a product of a symplectic manifold  and a Poisson manifold of rank zero at the origin. So an interesting topic is to explore the transverse Poisson structure along a symplectic leaf in the dual of a Lie algebroid. We shall leave the questions mentioned above to our future work.

The paper is organized as follows. In the second section, we clarify the adjoint and coadjoint representations of the  first jet groupoid on a Lie algebroid and its dual. In the third section, we realize symplectic leaves of the Poisson structure on the dual of a Lie algebroid as affine coadjoint orbits. The fiber bundle structure on a symplectic leaf is analyzed in the fourth section. At the last section, we understand symplectic leaves as the groupoid orbits of the symplectic groupoid $T^*\mathcal{G}$ over $A^*$. 

{\bf Notations}: A Lie algebroid is denoted by $(A,\rho,[\cdot,\cdot])$, where $\rho:A\to TM$ is the anchor map. We always suppose  $A$ is integrable with $\mathcal{G}$ as its  $r$-connected Lie groupoid. For simplicity,  $r$ and $l$ are used to denote the source and target of  any Lie groupoid respectively and $p$ is used to denote the projection of any vector bundle to its base manifold. Since we only consider the first jets, we call the first jet groupoid and the first jet algebroid as the jet groupoid and the jet algebroid in short.



\subsection*{Acknowledgments}
We would like to thank C. Laurent-Gengoux and K. Mackenzie for useful discussions and comments.
Honglei Lang is also grateful to the Max Planck Institute for Mathematics, Bonn, where she was staying while most of this project was done, for the financial support and excellent working conditions.

\section{Coadjoint representation of Lie groupoids}

For a Lie groupoid, there is no natural way to define an ordinary adjoint representation.  One approach is to construct a representation up to homotopy of the Lie groupoid $\mathcal{G}$ on the  normal complex $A\ominus TM$ \cite{Lu}, where $A\to M$ is the Lie algebroid of $\mathcal{G}$. See also \cite{ACG} for more detailed discussion.  On the other hand, 
unlike the Lie groupoid $\mathcal{G}$ itself, the  first jet groupoid $\mathcal{J}\mathcal{G}$ has a natural representation on the normal complex $A\ominus TM$ \cite{Lu, CSS}. We refer to \cite{Mac} for the general theory of Lie groupoids and algebroids. 

Here we shall first recall the definitions of adjoint and coadjoint representations given by the second idea. Then we discuss the mechanical properties of the coadjoint representation.

 Given a Lie groupoid $\mathcal{G}\rightrightarrows M$ and a surjective smooth map $J:P\to M$, a {\bf left action} of $\mathcal{G}$ on $P$ along the map $J$, called the moment map, is a smooth map
\[\mathcal{G}\times_M P\to P,\qquad (g,p)\mapsto g\cdot p\]
satisfying the following:
\[J(g\cdot p)=l(g),\qquad (gh)\cdot p=g\cdot (h\cdot p),\qquad 1_{J(p)}\cdot p=p,\]
where $\mathcal{G}\times_M P=\{(g,p)|r(g)=J(p)\}$.
Then $P$ is called a {\bf left $\mathcal{G}$-space}. 
 A {\bf representation} of a groupoid $\mathcal{G}\rightrightarrows M$ is a vector bundle $E$ over $M$ with a linear action of $\mathcal{G}$ on $E$, i.e. for each arrow $g: x\to y$, the induced map $g\cdot : E_x\to E_y$ on the fibers is a linear isomorphism.

Let $\mathcal{G}$ be a Lie groupoid with Lie algebroid $A$. A bisection of $\mathcal{G}$ is a splitting $b:M\to \mathcal{G}$ of the source map $r$ with the property that $\phi_b:=l\circ b:M\to M$ is a diffeomorphism. The bisections of $\mathcal{G}$ form a group $\mathrm{Bis}(\mathcal{G})$ with the multiplication and inverse given by
\[b_1\cdot b_2(x)=b_1(\phi_{b_2}(x))b_2(x),\qquad b^{-1}(x)=\mathrm{inv}\circ b\circ \phi_b^{-1}(x).\]
The identity is the inclusion of $M$ in $\mathcal{G}$. Local bisections are defined similarly. In particular, $\mathrm{LBis_x}(\mathcal{G})$, consisting of local bisections on open neiborhoods of $x\in M$, is a group. 

The  {\bf first jet groupoid} $\mathcal{J} \mathcal{G}$ of a Lie groupoid $\mathcal{G}$, consisting of $1$-jets $j_x^1 b$ of local bisections of $\mathcal{G}$ at $x$, is a Lie groupoid with 
the groupoid structure  given by
\[r(j^1_x b)=x,\quad \quad l(j^1_x b)=\phi_b(x),\qquad
j^1_{\phi_{b_2}(x)} b_1\circ j_x^1 b_2=j_x^1(b_1\cdot b_2), \quad \quad (j_x^1 b)^{-1}=j_{\phi_b(x)}^1 (b^{-1}).\]

A bisection $b$ of $\mathcal{G}$ acts on $\mathcal{G}$ by conjugation
\[\mathrm{AD}_b (g)=b(l(g))\cdot g\cdot b(r(g))^{-1},\]
which maps units to units and source fibers to source fibers. Its differential depends only on $j_x^1 b$ and thus defines representations of $\mathcal{J}\mathcal{G}$ on $A$ and $TM$. 
\begin{Pro}\label{CSS}\cite{CSS}
The jet groupoid $\mathcal{J}\mathcal{G}$  represents naturally on $A$ and $TM$ by 
\[\Ad_{j_x^1 b}:A_x\to A_{\phi_b(x)},\qquad \Ad_{j_x^1 b} u=dR_{b(x)^{-1}}dm_{(b(x),x)}(db(\rho(u)),u),\qquad u\in A_x,\]
where $m:\mathcal{G}\times \mathcal{G}\to \mathcal{G}$ is the groupoid multiplication, and 
\[\Ad_{j_x^1 b}:T_x M \to T_{\phi_b(x)} M,\qquad \Ad_{j_x^1 b} X=(d\phi_b)_x X,\qquad  X\in T_x M.\]
Moreover, the anchor $\rho: A\to TM$ is  equivariant:
\begin{eqnarray}\label{compatible}
\Ad_{j_x^1 b} \circ \rho =\rho\circ \Ad_{j_x^1 b}
\end{eqnarray}
for the above two representations.\end{Pro}

We call the representation $\Ad$ of $\mathcal{J}\mathcal{G}$ on $A$ the {\bf adjoint representation} of the Lie groupoid $\mathcal{G}$.
In particular, for $u\in \mathrm{ker}\rho_x$, we have
\[\Ad_{j_x^1 b}: \mathrm{ker} \rho_x\to \mathrm{ker} \rho_{\phi_b(x)},\qquad \Ad_{j_x^1 b} u=\Ad_{b(x)} u=dL_{b(x)} dR_{b(x)} (u).\]
The {\bf coadjoint representation} of $\mathcal{J}\mathcal{G}$ on $A^*$ is naturally defined as
\begin{eqnarray}\label{coadjoint}
\Ad^*_{j_x^1 b}: A_x^*\to A_{\phi_b(x)}^*,\qquad \langle \Ad^*_{j_x^1 b}  \alpha,u\rangle:=\langle \alpha, \Ad_{j_{\phi_b(x)}^1 b^{-1}} u\rangle, \qquad \alpha\in A_x^*, u\in A_{\phi_b(x)}.
\end{eqnarray}
Similarly, we have a representation of $\mathcal{J}\mathcal{G}$ on $T^*M$:
\[\Ad^*_{j_x^1 b}: T_x^* M\to T_{\phi_b(x)}^* M,\qquad \Ad^*_{j_x^1 b}(\gamma)=(d\phi_b^{-1})^*(\gamma),\qquad \gamma\in T_x^*M.\]
Let us have a look at the adjoint and coadjoint representation of the Lie groupoid $M\times G\times M\rightrightarrows M$, which is the direct product of the pair groupoid $M\times M$ and the Lie group $G$ and is also the gauge groupoid of the trivial principal $G$-bundle $M\times G$ over $M$.
\begin{Ex}
A local bisection $b$ of $M\times G\times M$ is a pair of maps $(f,l)$, where $f:U\subset M\to f(U)\subset M$ is a diffeomorphism, $l:U\to G$ is a smooth map, and $b(x)=(f(x),l(x),x)$.  We have $b^{-1}(x)=(f^{-1}(x),l(f^{-1}(x))^{-1},x),x\in f(U)$. Denote by 
$\tilde{l}:f(U)\to G$ the map $\tilde{l}(x)=l(f^{-1}(x))^{-1}$.

For  $(X,a)\in T_x M\times \mathfrak{g}$,  the adjoint representation is
 \begin{eqnarray*}
 \Ad_{j_x^1b}(X,a)&=&\frac{d}{dt}|_{t=0} \mathrm{AD}_b(\phi_t^X(x),e^{ta},x)=
 \frac{d}{dt}|_{t=0}\big(f(\phi_t^X(x)),l(\phi_t^X(x))e^{ta}l(x)^{-1},f(x))\\ &=&
(df(X), \Ad_{l(x)} a+dR_{l(x)^{-1}} dl(X),0).
\end{eqnarray*}
The coadjoint representation is thus
\begin{eqnarray*}
\langle \Ad_{j_x^1b}^*(\alpha,\mu),(Y,a)\rangle &=&
\langle (\alpha,\mu),(df^{-1}(Y), \Ad_{l(x)^{-1}} a+dR_{l(x)}d\tilde{l}(Y))\rangle\\ &=&
\langle (df^{-1})^*\alpha, Y\rangle+\langle \Ad_{l(x)}^* \mu,a\rangle+\langle \mu, dR_{l(x)}d\tilde{l}(Y)\rangle,
\end{eqnarray*}
for $(Y,a)\in T_{f(x)} M\times \mathfrak{g}$.
Notice that $\Ad_{j_x^1b}^*$ maps $\mu\in \mathfrak{g}^*$ not only to $\mathfrak{g}^*$ but also to  $T^*M$.
\end{Ex}

\begin{Pro}
Associated to the jet groupoid, there is a short exact sequence of groupoids
\[1\to \Hom(TM,A)^0\to \mathcal{J}\mathcal{G}\to \mathcal{G}\to 1.\]
Here $\Hom(TM,A)^0$ is defined as the space of units of $\Hom(TM,A)$ with  the multiplication 
\begin{eqnarray}\label{product}
\Phi\cdot \Psi=\Phi+\Psi-\Phi\circ \rho \circ \Psi,\quad \quad \Phi,\Psi\in \Hom(TM,A).                                      \end{eqnarray}
This is a Lie group bundle whose  infinitesimal  is the Lie algebra bundle
$(\Hom(TM,A),[\cdot,\cdot])$, where
\[[D,D']=D'\circ \rho\circ D-D\circ \rho \circ D',\qquad D,D'\in \Gamma(\Hom(TM,A)).\]
\end{Pro}
\begin{proof}
We sketch the proof here. The groupoid morphism from $\mathcal{J}\mathcal{G}$ to $\mathcal{G}$ is $j_x^1 b\mapsto b(x)$. For an element $\Phi\in \Hom(T_x M,A_x)$, let $b_1$ be any local bisection of $\mathcal{G}$ such that \[b_1(x)=1_x, \qquad (db_1)_x=id+\Phi: T_xM \to T_{x}\mathcal{G}=T_x M\oplus A_x.\]
$l\circ b_1$ being a diffeomorphism forces $\Phi$ to satisfy that $id+\rho\circ \Phi$ is invertible, which is equivalent to the condition that $\Phi$ is invertible with respect to the multiplication (\ref{product}) by direct calculation. This justifies the exactness of the sequence. 

For $\Phi,\Psi\in \Hom(T_x M,A_x)$, choose local bisections $b_1$ and $b_2$ such that $b_1(x)=b_2(x)=1_x$ and $(db_1)_x=id+\Phi$ and $(db_2)_x=id+\Psi$. The product of $b_1$ and $b_2$ is 
$(b_1\cdot b_2)(y)=b_1(\phi_{b_2}(y))b_2(y)$. Since $(d\phi_{b_1})_x=(dl\circ b_1)_x=id+\rho\circ \Phi$, we have
\[(d\phi_{b_1\cdot b_2})_x=(d(\phi_{b_1}\circ \phi_{b_2}))_x=(id+\rho\circ \Phi)\circ (id+\rho\circ \Psi)=id+\rho\circ (\Phi+\Psi+\Phi\circ \rho\circ \Psi).\]
This shows that the multiplication in $\Hom(TM,A)^0$ is (\ref{product}). 

It is routine to check that the Lie algebra of $\Hom(T_x M,A_x)^0$ is $(\Hom(T_x M,A_x),[\cdot,\cdot])$. We omit the proof here.
\end{proof}
For  a vector bundle $A$, its first jet bundle $\mathcal{J} A$ fits into the following short exact sequence of vector bundles
\[0\to \Hom(TM,A)\to \mathcal{J}A\to A\to 0.\]
Although it does not have a canonical splitting, at the level of sections, it does: $u\to \mathbbm{d} u$ for $u\in \Gamma(A)$. This gives an identification
\[\Gamma(\mathcal{J}A)\cong\Gamma(A)\oplus \Gamma(T^*M\otimes A).\]
So a section of $\mathcal{J}A$ is written as $\mathbbm{d} u+D$ for $u\in \Gamma(A)$ and $D\in \Gamma(T^*M\otimes A)$ and the $C^\infty(M)$-module structure becomes
\[f(\mathbbm{d}u+D)=\mathbbm{d}(fu)-df\otimes u+fD,\qquad u\in \Gamma(A), f\in C^\infty(M).\]
\begin{Pro}\label{omni-Lie algebroid}\cite{omni-Lie algebroid}
If $A$ is a Lie algebroid, then there is a Lie algebroid structure on $\mathcal{J}A$ with anchor  $\rho_{\mathcal{J} A}(\mathbbm{d}u+D)=\rho(u)$ and the bracket
\[[\mathbbm{d}u,\mathbbm{d}v]=\mathbbm{d}[u,v]_A,\quad [\mathbbm{d}u, D]=[u,D(\cdot)]_A-D[\rho(u),\cdot]_{TM},\quad[D,D']=-D\circ\rho\circ D'+D'\circ \rho\circ D.\]
Here $D,D'\in\Gamma(T^*M \otimes A)\cong \Gamma(\Hom(TM,A))$.
\end{Pro}
The jet algebroid $\mathcal{J}A$ is the Lie algebroid of the jet groupoid $\mathcal{J}\mathcal{G}$. We thus obtain the infinitesimal adjoint and coadjoint representations of $\mathcal{J}A$ on $A$ and $A^*$ respectively.
\begin{Lem}\label{infinitesimal actions}
We have that the adjoint representation $\mathrm{ad}: \mathcal{J} A\to \mathcal{D} A$ of $\mathcal{J}A$ on $A$  is given by
\begin{eqnarray}\label{inf adjoint}
\mathrm{ad}(\mathbbm{d} u)=[u, \cdot]_A, \qquad \mathrm{ad}(D)=-\rho^*(D),\qquad u\in \Gamma(A), D\in \Gamma(\Hom(TM,A)),
\end{eqnarray}
and the coadjoint representation of $\mathcal{J}A$ on $A^*$ is given by
\[(\mathbbm{d}u)\triangleright \alpha=\mathcal{L}_{u} \alpha,\qquad D\triangleright \alpha=i_\alpha\rho^*(D),\qquad \alpha\in \Gamma(A^*).\]
\end{Lem}
\begin{proof}
In general, an action of a Lie groupoid $\mathcal{G}$  on a vector bundle $E$ induces an action of its Lie algebroid $A$ on $E$:
\[\triangleright : \Gamma(A)\times \Gamma(E)\to \Gamma(E),\qquad u\triangleright e (x)=\frac{d}{dt}|_{t=0} \phi_t^u(x)^{-1}\cdot e(l\circ \phi_t^u(x))\in E_x,\qquad u\in \Gamma(A), e\in \Gamma(E),\]
where $\phi_t^u(x)\subset r^{-1}(x)$ is a flow of $u$. Moreover, there is an induced dual action of $\mathcal{G}$ on $E^*$ given by 
\[\langle g\cdot \xi(x),e(y)\rangle:=\langle \xi(x),g^{-1}\cdot e(y)\rangle,\qquad \xi\in \Gamma(E^*), g:x\to y.\]
Thus, by direct calculation, we get the induced dual action of $A$ on $E^*$:
\begin{eqnarray*}
\langle u\triangleright \xi,e\rangle(x)&=&\frac{d}{dt}|_{t=0}\langle \phi_t^u(x)^{-1}\cdot \xi(l\circ \phi_t^u(x)),e(x)\rangle=\frac{d}{dt}|_{t=0} \langle \xi(l\circ \phi_t^u(x)),\phi_t^u(x)\cdot e(x)\rangle\\ &=&\rho(u)\langle \xi,e\rangle (x)-\langle \xi, u\triangleright e\rangle(x).
\end{eqnarray*}
The results are direct from these formulas.
\end{proof}
Similarly, as $\mathcal{J}\mathcal{G}$ represents on $TM$, we have that
the Lie algebroid $\mathcal{J} A$ also represents on $TM$: 
\begin{eqnarray}\label{inf adjoint2}
(\mathbbm{d}u)\triangleright X=[\rho(u), X]_{TM},\qquad D\triangleright X=-i_X\rho(D),\qquad X\in \mathfrak{X}(M).
\end{eqnarray}
This induces a dual representation of $\mathcal{J} A$ on $T^*M$ given by
\begin{eqnarray}\label{dual1}
(\mathbbm{d}u)\triangleright \gamma =\mathcal{L}_{\rho(u)} \gamma,\qquad D\triangleright \gamma=i_\gamma \rho(D),\qquad \gamma\in \Omega^1(M),
\end{eqnarray}
where $\mathcal{L}_X: T^*M\to T^*M$ and $i_X: T^*M\otimes TM\to TM$ are the Lie derivative and contraction. 


Let us see some examples of jet algebroids and groupoids.

\begin{Ex}\label{jet groupoid}
The jet algebroid of a Lie algebra is the Lie algebra itself and the jet groupoid of a Lie group is the Lie group iteself; The jet algebroid of $TM$ is $\mathcal{D}(TM)$, the bundle of covariant differential operators, and the 
 jet groupoid of the pair groupoid $M\times M$ is the generalized linear groupoid $\mathrm{gl}(TM)$, whose arrows between two points $x,y\in M$ consist of  linear isomorphisms $T_x M \to T_yM$ \cite{Mac,omni-Lie algebroid}.

For  a transformation Lie algebroid $\mathfrak{g}\triangleright M$, it has a natural flat connection which is zero on the constant section. So its  first jet bundle splits into
 \[\mathcal{J} (\mathfrak{g}\times M)\cong\mathfrak{g}\times M\oplus T^*M\otimes \mathfrak{g},\]
 where $v\in \Gamma(\mathfrak{g}\times M)$ gives a section $\mathbbm{d}v-\nabla v\in \Gamma(\mathcal{J}(\mathfrak{g}\times M))$.
We then get a Lie algebroid structure on $\mathfrak{g}\times M\oplus T^*M\otimes \mathfrak{g}$ such that 
the the transformation Lie algebroid $\mathfrak{g}\triangleright M$ and the Lie algebra bundle $T^*M\otimes \mathfrak{g}$ are Lie subalgebroids and the mixed bracket is determined by  
\begin{eqnarray}\label{A on E}
[u,D]=[u,D(\cdot)]_{\mathfrak{g}}-D[\hat{u},\cdot]_{TM},\qquad u\in \mathfrak{g}, D\in \Gamma(T^*M\otimes \mathfrak{g}),
\end{eqnarray}
where $\hat{u}$ is the fundamental vector field of $u$ on $M$. Actually, this determines an action of $\mathfrak{g}\triangleright M$ on $T^*M\otimes \mathfrak{g}$ by derivations.

And the  first jet groupoid $\mathcal{J}(G\triangleright M)$ is isomorphic to
\[(G\times M)\times_M \Hom(TM,\mathfrak{g})^0\rightrightarrows M\]
with the groupoid structure 
\[(g,hx,\Psi)\cdot (h,x,\Phi)=(gh,x,(h,x)^{-1}\cdot \Psi+\Phi-((h,x)^{-1}\cdot \Psi)\circ \rho\circ \Phi),\]
where $(h,x)^{-1}\cdot \Psi$ is the action of $G\triangleright M$ on $T^*M\otimes \mathfrak{g}$ integrating (\ref{A on E}).  It has 
the transformation groupoid  $G\triangleright M$ and the group bundle $\Hom(TM,\mathfrak{g})^0$ (given by (\ref{product})) as Lie subgroupoids.

\end{Ex}

Similar to the Lie algebra case, a Lie algebroid $(A,[\cdot,\cdot],\rho)$ also gives a Poisson structure on its dual bundle $A^*$. Observe that the space of functions on $A^*$ are generated by two kinds of functions: $\Gamma(A)$ and $C^\infty(M)$. A section $u\in \Gamma(A)$ defines a function $l_u\in C^\infty(A^*)$ by $l_u(\alpha)=\langle u(x), \alpha\rangle, \forall \alpha\in A_x^*$.  A function $f\in C^\infty(M)$ gives naturally a function $p^* f$ on $A^*$, where $p: A^*\to M$ is the projection. The Poisson bracket on $A^*$ is given by
\[\{l_u,l_v\}=l_{[u,v]},\qquad \{l_u,p^* f\}=p^*(\rho(u)f),\qquad \{p^*f,p^*f'\}=0.\]
We call this Poisson structure the {\bf Lie-Poisson structure} on $A^*$.

For the Lie algebra case, we know the coadjoint action of the Lie group $G$ on $\mathfrak{g}^*$ preserves the symplectic structure on a symplectic leaf of the Lie Poisson structure on $\mathfrak{g}^*$ and it is further a Hamiltonian action. While for a Lie algebroid $A$,  the coadjoint action of the jet groupoid $\mathcal{J}\mathcal{G}$ does not always  preserve the Lie-Poisson structure on $A^*$ in the sense that the fundamental vector fields of $\Gamma(\mathcal{J} A)$ on $A^*$ are not always Poisson vector fields. Here  a Poisson vector field on a Poisson manifold $(M,\pi)$ is a vector field $X$ on $M$ such that $L_X \pi=0$.

\begin{Pro}\label{symp1}
The symplectic leaves of $A^*$ are invariant under the coadjoint action of the jet groupoid $\mathcal{J}\mathcal{G}$. That is, the coadjoint action is tangent to the symplectic leaves of $A^*$. Moreover, we have, for a section $\mathbbm{d}u+D$ of $\mathcal{J}A$,
\begin{itemize}
\item[(1)] its fundamental vector field on $A^*$ is a Poisson vector field iff
$D\circ \rho\in \mathrm{Der}(A)$, i.e.,\[D\circ \rho[u,v]=[D\circ \rho(u),v]+[u,D\circ \rho(v)],\qquad u,v\in \Gamma(A).\]
\item[(2)] its fundamental vector field on $A^*$ is a Hamiltonian vector field  iff $D\circ \rho=0$ and the corresponding Hamiltonian function is $l_u\in C^\infty(A^*)$.
\end{itemize}
\end{Pro}
\begin{proof}
For any $u\in \Gamma(A)$,  it is direct to see that the fundamental vector field  $\mathrm{ad}(\mathbbm{d}u)=[u,\cdot]_A \in \mathcal{D} A\subset \mathfrak{X}(A^*)$ is the Hamiltonian vector field $X_{l_u}$.  For $df\otimes u\in \Gamma(T^*M\otimes A)$,  the fundamental vector field  is \[\mathrm{ad}(df\otimes u)=\mathrm{ad}(\mathbbm{d}(fu))-\mathrm{ad}(f\mathbbm{d}u)=X_{fl_u}-fX_{l_u}.\]
It is not a Hamiltonian vector field in general, but at a point it is
 a linear combination of two Hamiltonian vectors. As any $D\in \Gamma(\Hom(TM,A))$ at a point is a linear combination of this form, we conclude that the fundamental vector field of any section of $\mathcal{J}A$ 
at a point is a Hamiltonian vector. So the coadjoint representation is tangent to the symplectic leaves of $A^*$.

The fundamental vector field of  $D\in \Gamma(\Hom(TM,A))$ is $\mathrm{ad}(D)=-D\circ \rho\in \End(A^*)\subset \mathfrak{X}(A)$. It is a Poisson vector field on $A^*$ iff $[\mathrm{ad}(D),\pi]=0$ iff $D\circ \rho\in \mathrm{Der}(A)$. It is further a Hamiltonian vector field iff $D\circ \rho=0$.
\end{proof}
It is seen from the proof that $df\otimes u$ is a Poisson vector field iff $\rho^*(df)\wedge \ad_u\in \mathfrak{X}^2(A^*)$ vanishes and a Hamiltonian vector field iff $\rho^*(df)=0$ iff $p^*f$ is a Casimir function on $A^*$.



Since the adjoint representation of $\mathcal{J}\mathcal{G}$ on $A^*$ is linear, it preserves the zero section. For the particular case $A=TM$ ($M$ connected), the only simplectic leaf of $T^*M$ is the entire manifold. So the coadjoint orbits of $\mathcal{J}\mathcal{G}$ on $A^*$ are not exactly the symplectic leaves.
We shall use the vertical translations to extend the coadjoint orbits to symplectic leaves. This is the content of the following section.
\section{Coadjoint orbits and symplectic leaves}
It is well-known that a symplectic leaf for the Lie-Poisson structure on the dual of a Lie algebra is a coadjoint orbit of its Lie group. We expect a parallel result for Lie algebroids. In fact, we realize symplectic leaves on the dual of a Lie algebroid as affine coadjoint orbits.



Since $\mathcal{J}\mathcal{G}$ represents on $T^*M$, we have the semi-direct product Lie groupoid $\mathcal{J}\mathcal{G}\ltimes T^*M$ over $M$. As a manifold, it is the fiber product $\mathcal{J}\mathcal{G}\times_{M} T^*M$ with respect to the source map $r:\mathcal{J}\mathcal{G}\to M$ and the projection $p: T^*M\to M$. The groupoid structure is given by
\[r(j^1_x b, \gamma)=x,\quad \quad l(j^1_x b,\gamma)=\phi_b(x), \qquad \gamma\in T_x^*M,\]
\[(j^1_{\phi_{b_2}(x)} b_1, \gamma_1)\circ (j_x^1 b_2,\gamma_2)=(j_x^1(b_1\cdot b_2),  \gamma_2+(d\phi_{b_2})^{*}\gamma_1),\qquad \gamma_1\in T_{\phi_{b_2}(x)}^*M, \gamma_2\in T_x^* M.\]
The jet groupoid  $\mathcal{J}\mathcal{G}$ coadjoint represents on $A^*$ by (\ref{coadjoint}). We also notice that
the trivial Lie groupoid $T^*M$ ($l=r=p$)  acts  on $A^*$ by the vertical translation
\begin{eqnarray}\label{vertical}
\gamma\triangleright \alpha=\alpha-\rho^*(\gamma),\quad \quad \gamma\in T_x^*M, \alpha\in A^*_x.
\end{eqnarray}

\begin{Thm}\label{main2}
With the notations above, we have
\begin{itemize}
\item[(1)]
the Lie groupoid $\mathcal{J}\mathcal{G}\ltimes T^*M$ acts on $A^*$ by 
\begin{eqnarray}\label{affine}
(j_x^1 b,\gamma)\triangleright \alpha= \Ad_{j_x^1 b}^* \alpha-\rho^*(d\phi_{b}^{-1})^*\gamma,\qquad \gamma\in T^*_x M,\alpha\in A^*_x,
\end{eqnarray}
with the moment map the natural projection $A^*\to M$. We call this action the {\bf  affine coadjoint action} of $\mathcal{G}$ on $A^*$.
\item[(2)] the orbits of the affine coadjoint action are the symplectic leaves of the Lie-Poisson structure on $A^*$.  
\end{itemize}
\end{Thm}
\begin{proof}
First we check $A^*$ is a $\mathcal{J}\mathcal{G}\ltimes T^*M$-space. For $\gamma_2\in T^*_x M$ and $\gamma_1\in T_{\phi_{b_2}(x)}^* M$, we have 
\begin{eqnarray*}
(j^1_{\phi_{b_2}(x)} b_1, \gamma_1)\triangleright\big( (j_x^1 b_2,\gamma_2)\triangleright \alpha\big)&=&(j^1_{\phi_{b_2}(x)} b_1, \gamma_1)\triangleright (\Ad_{j_x^1 b_2}^* \alpha-\rho^*(d\phi_{b_2}^{-1})^* \gamma_2)\\ &=&\Ad^*_{j^1_{\phi_{b_2}(x)} b_1} \Ad^*_{j^1_x b_2} \alpha-\Ad^*_{j^1_{\phi_{b_2}(x)} b_1} \rho^*(d\phi_{b_2}^{-1})^*\gamma_2-\rho^*(d\phi_{b_1}^{-1})^* \gamma_1\\ &=&\Ad_{j_x^1(b_1\cdot b_2)}^* \alpha-\rho^*(d\phi_{b_1}^{-1})^*(d\phi_{b_2}^{-1})^* \gamma_2-\rho^*(d\phi_{b_1}^{-1})^{*}\gamma_1\\ &=&
(j_x^1(b_1\cdot b_2),  \gamma_2+(d\phi_{b_2})^{*}\gamma_1)\triangleright \alpha\\ &=&
\big((j^1_{\phi_{b_2}(x)} b_1, \gamma_1)\circ (j_x^1 b_2,\gamma_2)\big)\triangleright \alpha.
\end{eqnarray*}
Here we have used  the equation \eqref{compatible} and the fact that $\rho^*(d\phi_{b_1\cdot b_2}^{-1})^*(d\phi_{b_2})^*\gamma_1=\rho^*(d\phi_{b_1}^{-1})^* \gamma_1$.

For $(2)$, we show that the vector space of fundamental vector fields of the affine coadjoint action of $\mathcal{J}\mathcal{G}\ltimes T^*M$ on $A^*$ at $\alpha\in A^*_x$ coincides with the vector space of all Hamiltonian vector fields of $A^*$ at $\alpha$.

For $u\in \Gamma(A)$ and $f\in C^\infty(M)$, the Hamiltonian vector fields of $l_u, p^*f\in C^\infty(A^*)$ are respectively
\[ X_{l_u}=\ad_u=[u,\cdot]_A\in \mathcal{D}(A),\quad \quad X_{p^*f}=-\rho^*(df)\in \Gamma(A^*).\]
And the fundamental vector fields are 
\[\widehat{\mathbbm{d}u}=\ad_u\in \mathcal{D}(A),\quad \quad \widehat{df\otimes u}=
\ad_{fu}-f\ad_u\in \End(A),\quad \quad \widehat{\gamma}=-\rho^*\gamma\in \Gamma(A^*),\qquad \gamma\in \Omega^1(M).\]
Locally, any $1$-form $\gamma\in \Omega^1(M)$ can be written as $\gamma(x)=\sum_{i=1}^{\mathrm{dim}(M)} \gamma^i(x) dx^i$, the summation of exact forms at this point. Also note that $X_{fg}=gX_f+fX_g$. We obtain that the vector spaces of fundamental vector fields and Hamiltonian vector fields are the same at a point.
Since the affine coadjoint orbits of $\mathcal{J}\mathcal{G}\ltimes T^*M$ are connected, we get that the symplectic leaves of the Lie-Poisson structure on $A^*$ are exactly the affine coadjoint orbits.  
\end{proof}

From the above proof, one see the space of fundamental vector fields of $\Gamma(A)$ and $\Omega^1(M)$ already coincides with the space of Hamiltonian vector fields on $A^*$ at a point. This means that the integrable leaves generated by the actions of $\Gamma(A)$ and $\Omega^1(M)$ are the symplectic leaves. Here  in order to realize the symplectic leaves as orbits of an action,  we also take $\Gamma(T^*M\otimes A)$ into consideration.

We have proved that the symplectic leaves passing a point $\alpha\in A_x^*$ is 
\begin{eqnarray}\label{sym leaf}
\mathcal{O}_\alpha=\{\Ad_{j_x^1 b}^* \alpha+\rho^*(d\phi_{b}^{-1})^*(\gamma);\qquad \forall \gamma\in T_x^* M, b\in \mathrm{LBis_x(\mathcal{G})}\}.
\end{eqnarray}
\begin{Cor}\label{3.2}
The projection of the symplectic leaf $\mathcal{O}_\alpha$ for $\alpha\in A_x^*$ to $M$ is the groupoid orbit $\mathcal{L}_x$ in $\mathcal{G}$ passing $x$.
\end{Cor}
\begin{proof}
We shall check $\mathcal{L}_x=\{\phi_b(x),  b\in \mathrm{LBis_x}(\mathcal{G})\}$. For any local bisection $b$, we have $\phi_b(x)=l(b(x))\in l(r^{-1}(x))=\mathcal{L}_x$. Conversely for any $y\in \mathcal{L}_x$, there exists $g\in \mathcal{G}$ such that $l(g)=y$ and $r(g)=x$. Choosing a local bisection $b$ passing $g$, we get $\phi_{b}(x)=y$. 
\end{proof}
So the symplectic leaf $\mathcal{O}_\alpha$ through $\alpha\in A_x^*$ only depends on the Lie algebroid $A$ restricting on $\mathcal{L}_x$.

\begin{Cor}
The Lie-Poisson structure on $A^*$ is symplectic iff $A^*$ is isomorphic to $T^*M$ with the canonical symplectic structure.
\end{Cor}

\begin{Rm}
Here we consider the symplectic leaves of the Lie-Poisson structure on $A^*$, which is based on the Weinstein splitting theorem for Poisson manifolds. The splitting theorem for  a  Lie algebroid $A$  is also well-studied \cite{F, D}, which is not the Weinstein splitting  theorem for the Lie-Poisson structure on $A^*$ as mentioned by Fernandes in \cite{F}. The reason is that for the Lie algebroid case changes of coordinates of $A^*$ are only allowed to be linear in the fiber variables.
\end{Rm}

As $\mathcal{J}A$ represents on $TM$, we get the semi-direct product Lie algebroid $\mathcal{J} A\ltimes T^*M$ with the bracket between $\Gamma(\mathcal{J}A)$ and $\Omega^1(M)$ given by the representation (\ref{dual1}). 

\begin{Pro}
The Lie algebroid of the Lie groupoid $\mathcal{J}\mathcal{G}\ltimes T^*M$ is $\mathcal{J}A \ltimes T^*M$ and the  infinitesimal of the affine coadjoint action of $\mathcal{J}\mathcal{G}\ltimes T^*M$ on $A^*$ is the Lie algebra homomorphism
\begin{equation}\label{affine action}
\mathrm{ad}: \Gamma(\mathcal{J}A \ltimes T^*M)\to \mathfrak{X}(A^*),\quad \quad \mathrm{ad}(\mathbbm{d}u+D+\gamma)=[u,\cdot]_A-\rho^*(D)-\rho^*(\gamma),
\end{equation}
for $u\in \Gamma(A), D\in \Gamma(\Hom(TM,A))$ and $\gamma\in \Omega^1(M)$. 
\end{Pro}
\begin{proof}
For $\gamma\in \Omega^1(M)$, we have $x\mapsto (j^1_x i, t\gamma_x)\in \mathcal{J}\mathcal{G}\ltimes T^*M$, where $i:M\hookrightarrow G$ is the inclusion,  is a bisection integrating $\gamma$. Then the induced infinitesimal action of $T^*M $ on $A^*$ is 
\[\ad(\gamma)(\alpha)=\frac{d}{dt}|_{t=0} (j^1_x i, t\gamma_x)\triangleright \alpha=-\rho^*(\gamma_x), \qquad \alpha\in A_x^*.\]
With Lemma \ref{infinitesimal actions}, we complete the proof.
\end{proof}


In the following, we shall realize the affine coadjoint action of $\mathcal{J}A\ltimes T^*M$ on $A^*$ as 
a restriction of the adjoint representation of an extended Lie algebroid on itself. 

Let $\tilde{A}:=A\ltimes (M\times \mathbb{R})$ denote the semi-direct product Lie algebroid  of the Lie algebroid $A$ with the action $u\vartriangleright f=\rho(u)f, \forall u\in \Gamma(A), f\in C^\infty(M)$. Consider the adjoint representation
\[\tilde{\mathrm{ad}}:\mathcal{J}\tilde{A}\to \mathcal{D} \tilde{A}\]
of $\mathcal{J}\tilde{A}$ on $\tilde{A}$ which is given by 
\begin{eqnarray}\label{linear action1}
\tilde{\mathrm{ad}}(\mathbbm{d}(u+f))(v+g)=[u+f,v+g]_{\tilde{A}}=[u,v]+\rho(u)g-\rho(v)f
\end{eqnarray}
and 
\begin{eqnarray}\label{linear action2}
\tilde{\mathrm{ad}}(D+\gamma)(v+g)=-\tilde{\rho}^*(D+\gamma)(v+g)=-D(\rho(v))-\gamma(\rho(v)),
\end{eqnarray}
for $u,v\in \Gamma(A),f,g\in C^\infty(M), D\in \Gamma(\Hom(TM,A))$ and $\gamma\in \Omega^1(M)$.


\begin{Pro}
\begin{itemize}
\item[(1)] We have that $\mathcal{J}\tilde{A}=(\mathcal{J} A\ltimes T^*M)\bowtie (M\times \mathbb{R})$ is a matched pair of Lie algebroids, that is, $\mathcal{J}A\ltimes T^*M$ and $M\times \mathbbm{R}$ are  sub-algebroids of $\mathcal{J}\tilde{A}$ and the mixed bracket is 
\[[\mathbbm{d}u+D+\gamma, f]=\rho(u) f+i_{df} \rho(D).\]
\item[(2)] The affine action $\pi$ of $\mathcal{J}A\ltimes T^*M$ on $A^*$ is the restriction of the adjoint representation $\tilde{\ad}$   to $\mathcal{J} A\ltimes T^*M\subset \mathcal{J} \tilde{A}$ on $A^*\subset \tilde{A^*}$.
\end{itemize}
\end{Pro}

\begin{proof}
As vector bundles, we have $\mathcal{J}\tilde{A}=\mathcal{J}A \oplus (T^*M\oplus M\times \mathbbm{R})$. By Proposition \ref{omni-Lie algebroid},  we have
\[[\mathbbm{d}u+D,\gamma+f]=L_{\rho(u)} \gamma+\rho(u)f+i_\gamma \rho(D)+i_{df} \rho(D).
\]
Comparing with the equation (\ref{dual1}), we see that $\mathcal{J} A\ltimes T^*M$ is a Lie sub-algebroid of $\mathcal{J}\tilde{A}$ and also the other results.
(2) follows by comparing the formula (\ref{affine action})  with (\ref{linear action1}) and (\ref{linear action2}).

\end{proof}
\begin{Rm}
This description justifies the minus sign we put in (\ref{vertical}) and then in $\ad(\gamma)=-\rho^*(\gamma)$ for $\gamma\in \Omega^1(M)$ of (\ref{affine action}). While $\ad(D)=-\rho^*(D)$ is forced by $\pi(\mathbbm{d}u)=[u,\cdot]_A$ for $u\in \Gamma(A)$.
\end{Rm}


For the affine coadjoint action of $\mathcal{J} \mathcal{G}\ltimes T^*M$ on $A^*$ defined by (\ref{affine}),
we have known from Proposition \ref{symp1} that the first part $\mathcal{J} \mathcal{G}$ does not preserve the Lie-Poisson structure on $A^*$. For the $T^*M$ part,
we see from (\ref{vertical}) that the trivial Lie groupoid $T^*M$ acts on $A^*$ with the infinitesimal  $\mathrm{ad}: \Omega^1(M)\to \mathfrak{X}(A^*), \ad(\gamma)=-\rho^*(\gamma)$. We shall study when the fundamental vector fields are Poisson and Hamiltonian vector fields.
\begin{Pro}\label{symp2}
For $\gamma\in \Omega^1(M)$, its fundamental vector field $\ad(\gamma)$ is a Poisson vector field iff $\rho^*(d\gamma)\in \Gamma(\wedge^2 A^*)$ vanishes; it is a Hamiltonian vector field iff $\rho^*(\gamma)=\rho^*(df)$ for some $f\in C^\infty(M)$.

In particular, if $A$ is transitive, then $\mathrm{ad}(\gamma)$ for any $\gamma\in \Omega^1(M)$ is a Poisson vector field on $A^*$ iff $\gamma$ is  closed and a Hamiltonian vector field iff $\gamma$ is exact.
\end{Pro}
\begin{proof}
The vector field $-\rho^*(\gamma)\in \Gamma(A^*) $ is a Poisson vector field iff it is a derivation with respect to the Poisson bracket on $A^*$. 
First by definition we have
\[\rho^*(\gamma)(l_u)=p^*\langle \gamma,\rho(u)\rangle, \qquad \rho^*(\gamma)(p^*f)=0,\qquad u\in \Gamma(A),f\in C^\infty(M),\] where $p:A\to M$ is the projection. Then for $u,v \in \Gamma(A)$, we have
\begin{eqnarray*}
\rho^*(\gamma)\{l_u,l_v\}-\{\rho^*(\gamma)(l_u),l_v\}-\{l_u,\rho^*(\gamma)(l_v)\} &=&p^*\big(\langle \gamma,\rho[u,v]\rangle+\rho(v)\langle \gamma, \rho(u)\rangle-\rho(u)\langle \gamma,\rho(v)\rangle\big)\\ &=&-p^*(d\gamma(\rho(u),\rho(v))).
\end{eqnarray*}
Moreover, it is easy to see 
\begin{eqnarray*}
&&\rho^*(\gamma)\{l_u,p^*f\}-\{\rho^*(\gamma)(l_u),p^*f\}-\{l_u,\rho^*(\gamma)(p^*f)\}=0,\\ &&\rho^*(\gamma)\{p^*g,p^*f\}-\{\rho^*(\gamma)(p^*g),p^*f\}-\{p^*g,\rho^*(\gamma)(p^*f)\}=0.
\end{eqnarray*}
Hence we obtain that $\rho^*(\gamma)$ is a Poisson vector field iff $\rho^*(d\gamma)=0$. Furthermore, it is direct to see $\rho^*(\gamma)$ is a Hamiltonian vector field iff $\rho^*(\gamma)=\rho^*(df)$ for some $f\in C^\infty(M)$. Here we have used the fact $p:A\to M$ is a surjection. We then get the result.
\end{proof}

In the Lie algebra case, a symplectic leaf is a symplectic homogeneous space. For a Lie algebroid, it is seen from Proposition \ref{symp1} and \ref{symp2} that in general the affine coadjoint action of $\mathcal{J}\mathcal{G}\ltimes T^*M$ does not always preserve the symplectic structure on a symplectic leaf. We shall show that a symplectic leaf in the Lie-Poisson manifold $A^*$ is still a homogeneous space of Lie groupoids.




Recall from \cite{Liu} that a  $\mathcal{G}$-space  $P$ over $M$ is {\bf homogeneous} if there is a section $\sigma$  of the moment map $J: P\to M$ satisfying $\mathcal{G}\cdot \sigma(M)=P$. The {\bf isotropy subgroupoid} of the section $\sigma$ consists of those $g\in \mathcal{G}$ for which $g\cdot \sigma(M)\subset \sigma(M)$. This definition makes sure that a $\mathcal{G}$-space  is homogeneous if and only if it is isomorphic to $\mathcal{G}/\mathcal{H}$ for some wide subgroupoid $\mathcal{H}\subset \mathcal{G}$.

For $\alpha\in A_x^*$,  we get that $\mathcal{O}_{\alpha}$ is a homogeneous $\mathcal{J}\mathcal{G}\ltimes T^*M|_{\mathcal{L}_x}$-space since $\mathcal{J}\mathcal{G}\times T^*M \cdot \alpha=\mathcal{O}_{\alpha}$. We can choose an arbitrary section $\sigma$ of the projection $\mathcal{O}_\alpha \to \mathcal{L}_x$ such that $\sigma(x)=\alpha$. Moreover, we have 
\[\mathcal{O}_{\alpha}=(\mathcal{J}\mathcal{G}\ltimes T^*M|_{\mathcal{L}_x})/\mathcal{H}, \]
where $\mathcal{H}$ is the isotropy groupoid of $\sigma$.

\begin{Ex}
For $A=TM$, the isotropy groupoid  of a $1$-form $\sigma\in \Omega^1(M)$ extending $\alpha\in T_x^*M$ is as follows:
\[\mathcal{H}=\{(j_x^1 b, -\sigma(\phi_b(x))+(d\phi_b^{-1})^* \alpha)\in \mathcal{J}\mathcal{G}\ltimes T^*M\}\cong \mathcal{J}\mathcal{G},\]
where $b:M\to M\times M: x\to (\phi_b(x),x)$ is a bisection of $\mathcal{G}=M\times M$. In this case, we get the map
\[\psi: \mathcal{J}\mathcal{G}\ltimes T^*M\to \mathcal{J}\mathcal{G}\ltimes T^*M/\mathcal{H}\xrightarrow{\cong} T^*M,\quad \quad  (j_x^1b,\gamma)\to [(j_x^1 b,\gamma)]\to (d \phi_b^{-1})^*\alpha-(d\phi_b^{-1})^* \gamma.\]
\end{Ex}

\section{Geometric structures of  symplectic leaves}
\subsection{Symplectic leaves and symplectic reduced spaces}
As an application of Theorem \ref{main2}, we shall show that
a symplectic leaf $\mathcal{O}_\alpha$ in $A^*$ is actually a Marsden-Weinstein symplectic reduced space.

Let $\mathcal{G}\rightrightarrows M$ be the $r$-connected Lie groupoid of  a Lie algebroid $A$. For any $x\in M$, denote by $\mathcal{L}_x$ the algebroid leaf (groupoid orbit) passing $x$ and $G_x$ the isotropy Lie group at $x$.
By \cite[Theorem 5.4]{MM},  we know that for any $x\in M$, $l: r^{-1}(x)\to \mathcal{L}_x$ is a principal $G_x$-bundle,  whose gauge groupoid $r^{-1}(x)\times r^{-1}(x)/G_x$ is isomorphic to the Lie groupoid $\mathcal{G}|_{\mathcal{L}_x}$ under the map 
\[r^{-1}(x)\times r^{-1}(x)/G_x\to \mathcal{G}|_{\mathcal{L}_x},\qquad [h,g]\mapsto hg^{-1}.\]
\begin{Lem} \label{gauge}
One has a Lie algebroid isomorphism between the gauge Lie algebroid and $A$ restricting on $\mathcal{L}_x$:
\[Tr^{-1}(x)/G_x\to A|_{\mathcal{L}_x},\qquad [X]\mapsto dR_{g^{-1}}(X),\qquad X\in T_g r^{-1}(x).\]
The induced isomorphism on the isotropy Lie algebra bundle $\mathrm{ker} \rho|_{\mathcal{L}_x}$ is
\[r^{-1}(x)\times_{G_x} \mathrm{ker} \rho_x\cong \mathrm{ker} \rho |_{\mathcal{L}_x},\qquad [g,a]\mapsto \mathrm{Ad}_g a.\]
\end{Lem}

For simplicity of notations, we first talk about the gauge Lie algebroid $TP/G$ of a general principal $G$-bundle $P$ over $M$. Consider the Atiyah sequence \[0\to P\times_G \mathfrak{g}\xrightarrow {\widehat{\cdot}} TP/G\xrightarrow{\pi_*} TM\to 0,\]
and its dual sequence 
\begin{eqnarray}\label{Atiyah}
0\to T^*M \to T^*P/G\xrightarrow{J} P\times_G \mathfrak{g}^*\to 0.
\end{eqnarray}
By definition, $J$ is given by 
\[\langle J(\alpha_p), [p,u]\rangle:=\langle \alpha, \hat{u}\rangle (p),\quad \quad \alpha\in \Omega^1(P)^G,  u\in \mathfrak{g},\]
where $\hat{u}\in \mathfrak{X}(P)$ is the fundamental vector field generated by $u$.

Denote by $\tilde{J}: T^*P\to \mathfrak{g}^*$ the moment map of  the Hamiltonian action $G$ acting on $T^*P$ (cotangent lifting). Then one has $\tilde{J}(\alpha\cdot g)=\Ad_{g^{-1}}^*\tilde{J}(\alpha)$. The maps $\tilde{J}$ and $J$ in (\ref{Atiyah}) are related by 
 the following commutative diagram
\begin{equation*}
		\begin{tikzcd}
			T^*P\arrow{d}{} \ar{r}{\tilde{J}} &P\times \mathfrak{g}^*\arrow{d}{}\\
			T^*P/G \arrow{r}{J} &P\times_G \mathfrak{g}^*,
		\end{tikzcd}
\end{equation*}
where the vertical maps are the natural projections.

The following result is well-known; see \cite{Weinstein1983} for example. Here we prove it by applying Theorem \ref{main2}.
\begin{Pro}\label{MW-reduction}
For $\alpha\in T^*P/G$ such that $J(\alpha)=[p,\mu]\in P\times_G \mathfrak{g}^*$, 
we have \[\mathcal{O}_{\alpha}= \tilde{J}^{-1}(\mathcal{S}_\mu)/G.\]
Namely, a symplectic leaf is an orbit reduced space.
\end{Pro}
\begin{proof}
Notice that the bisection group of $P\times P/G$ is the space of $G$-invariant functions on $P$, namely,
\[\mathrm{Bis}(\mathcal{G})=\Aut(P)=\{F\in \mathrm{Diff}(P); F(pg)=F(p)g\}.\]
By Theorem \ref{main2} and straightforward calculation of the affine coadjoint orbits of this case, we find that the symplectic leaf passing $\alpha\in (T^*P/G)_x$ is 
\[\mathcal{O}_{\alpha}=\{(dF)^*(\alpha)+\pi_P^*(df)^*\gamma;\quad \forall \gamma\in T_{x}^*M, \forall F\in \mathrm{LBis_x}(\mathcal{G})\},\]
where $\pi_P$ is the projection from $P$ to $M$ and $f\in \mathrm{Diff}(M)$ is the unique function induced by $F$.

By definition, for a local bisection $F\in \mathrm{LBis_x}(\mathcal{G})$, we have 
\begin{eqnarray*}
\langle J((dF)^* \alpha+\pi_p^*((df)^*\gamma)), [F^{-1}(p),a]\rangle=\langle \alpha, dF(\hat{a}_{F^{-1}(p)})\rangle=\langle \alpha, \hat{a}_p\rangle=\langle J(\alpha),[p,a]\rangle=\langle \mu, a\rangle.
\end{eqnarray*}
This implies 
\[J((dF)^* \alpha+\pi^*(df)^*\gamma)=[F^{-1}(p),\mu]=[F^{-1}(p)g, \Ad_{g^{-1}}^* \mu].\]
So we have $J(\mathcal{O}_{\alpha})=[q,\mathcal{S}_\mu]$ for any $q\in P$. We thus get $\mathcal{O}_{\alpha}\subset \tilde{J}^{-1}(\mathcal{S}_\mu)/G$, where the latter is the symplectic reduced space. The connectedness forces them to be equal. This concludes the proof.
\end{proof}
It is shown in \cite{GS} that the symplectic reduced space $\tilde{J}^{-1}(\mu)/G_\mu$ is symplectically diffeomorphic to the orbit reduced space $\tilde{J}^{-1}(\mathcal{S}_\mu)/G$ and from this it follows that the symplectic leaf in $T^*P/G$ is a symplectic reduced space by Marsden and Weinstein.

As we have noticed from Corollary \ref{3.2},  the symplectic leaf $\mathcal{O}_\alpha$ passing  $\alpha\in A_x^*$ only depends on the Lie algebroid $A|_{\mathcal{L}_x}$, which is isomorphic to a gauge Lie algebroid by Lemma \ref{gauge}. Then by Proposition \ref{MW-reduction}, we have
\begin{Cor}
The symplectic leaf $\mathcal{O}_\alpha$ in $A^*$ is symplectically isomorphic to a  symplectic reduced space.
\end{Cor}

\begin{Rm}
In the case of gauge algebroids, we recover the symplectic reduction for cotangent bundles \cite{MR}.  And a symplectic leaf is a universal phase space  for a classical particle in a Yang-Mills field explained in \cite{Weinstein}. "Universal" means that the symplectic structure is independent of the choice of a connection, which is relative to the phase space constructed by Sternberg \cite{GS}.
\end{Rm}

\subsection{Fiber bundle structures of symplectic leaves}
Now we study the geometric structure on a symplectic leaf of $A^*$ for a Lie algebroid $A$.

For $x\in M$, denote by $\mathrm{ker} \rho_x$, the kernel  of the anchor map $\rho:A_x\to T_x M$ at $x$, which is called the {\bf isotropy Lie algebra}, and $\mathcal{L}_x$, the algebroid leaf through $x$ which is tangent to $\mathrm{Im} \rho$. The isotropy Lie group is denoted as $G_x$. There is a natural map $i_x^*: A_x^*\to (\mathrm{ker} \rho_x)^*$ that dualizes the inclusion $i_x: \mathrm{ker} \rho_x \hookrightarrow A_x$. So any $\alpha\in A_x^*$ defines an element $i_x^* \alpha$ in $(\mathrm{ker} \rho_x)^*$. As a consequence, to any $\alpha\in A^*_x$, we can associate two symplectic manifolds with canonical symplectic structures:
\begin{itemize}
\item[1)] $\mathcal{S}_{i^*\alpha}$, the coadjoint orbit of $i^*\alpha$ in the dual $(\mathrm{ker} \rho_x)^*$;
\item[2)] $T^*\mathcal{L}_x$, the cotangent bundle of the algebroid leaf through $x$.
\end{itemize}
The symplectic leaf $\mathcal{O}_\alpha$ passing $\alpha\in A_x^*$  relates with both. 

Let $r^{-1}(x)^\sharp$ be the pull-back  bundle of the principal $G_x$-bundle $r^{-1}(x)$ over $\mathcal{L}_x$ along the projection $T^*\mathcal{L}_x\to \mathcal{L}_x$:
\begin{equation*}
		\begin{tikzcd}
			r^{-1}(x)^\sharp\arrow{d}{} \ar{r}{} &r^{-1}(x)\arrow{d}{}\\
			T^*\mathcal{L}_x \arrow{r}{} &\mathcal{L}_x,
		\end{tikzcd}
\end{equation*}
which is a principal $G_x$-bundle  over $T^*\mathcal{L}_x$. 
\begin{Thm}\label{main}
The symplectic leaf passing $\alpha\in A_x^*$ is a fiber bundle over $T^*\mathcal{L}_x$  with the fiber type 
$\mathcal{S}_{i^*\alpha}$. 

Explicitly, it is isomorphic to the associated fiber bundle $r^{-1}(x)^\sharp\times_{G_x} \mathcal{S}_{i^*\alpha}$ over $T^*\mathcal{L}_x$.
\end{Thm}
\begin{proof} Taking the dual of the following exact sequence of vector bundles
\[0\to \ker \rho|_{\mathcal{L}_x}\xrightarrow{i} A|_{\mathcal{L}_x}\xrightarrow{\rho}  T\mathcal{L}_x\to 0,\]
we get
\[0\to T^*\mathcal{L}_x\xrightarrow{\rho^*} A^*|_{\mathcal{L}_x}\xrightarrow{i^*} (\mathrm{ker} \rho)^*|_{\mathcal{L}_x}\to 0.\]

Choose a horizontal lifting $\lambda:T\mathcal{L}_x\to A|_{\mathcal{L}_x}$, i.e. a bundle map such that $\rho\circ \lambda=id$.  We claim the map
\[\lambda^*:\mathcal{O}_\alpha\to T^*\mathcal{L}_x,\]
 is a surjection and such that the fiber over a point in $T_y^* \mathcal{L}_x$ is a coadjoint orbit of $G_y$  on $(\mathrm{ker} \rho_y)^*$.
In fact, for any $\beta\in T_y^*\mathcal{L}_x$,  let $b_0$ be any local bisection such that $\phi_{b_0}(x)=y$. Choose $\gamma\in T_x^*M$ such that $j^*((d \phi_{b_0}^{-1})^*\gamma)=\beta-\lambda^*(\Ad_{j_x^1 {b_0}}^* \alpha)$, where $j:\mathcal{L}_x\hookrightarrow M$ is the inclusion map.  Then we have
\[\lambda^*(\Ad_{j_x^1 b_0}^* \alpha+\rho^*(d\phi_{b_0}^{-1})^*\gamma)=\beta.\]
By (\ref{sym leaf}), we see $\Ad_{j_x^1 b_0}^* \alpha+\rho^*(d\phi_{b_0}^{-1})^*\gamma\in \mathcal{O}_\alpha$ and thus the map $\lambda^*:\mathcal{O}_\alpha\to T^*\mathcal{L}_x$ is onto.

For any $a\in \mathrm{ker} \rho_y$ and $y\in \mathcal{L}_x$,  we have
\begin{eqnarray*}
\langle i^*(\Ad_{j_x^1 b_0}^* \alpha+\rho^*(d\phi_{b_0}^{-1})^*\gamma), a\rangle&=&\langle \Ad_{j_x^1 b_0}^*\alpha, i(a)\rangle=\langle \alpha, \Ad_{b_0(y)^{-1}} i(a)\rangle \\ &=&\langle \alpha, i(\Ad_{j_y^1 b_0^{-1}} a)\rangle =\langle \Ad_{b_0(x)}^* i^*\alpha,a\rangle.
\end{eqnarray*}
For any local bisection $b$ such that $\phi_b(x)=y$, we have $h:=b(x)b_0(x)^{-1}\in G_y$, the isotropy group at $y$.  Thus we have
\[i^*(\Ad_{j_x^1 b}^* \alpha+\rho^*(d\phi_{b}^{-1})^*\gamma)=\Ad_{h}^*(\Ad_{b_0(x)}^*i^* \alpha).\]
This implies that  the image of the symplectic leaf $\mathcal{O}_\alpha$ in $A^*$  under $i^*_y$ is the coadjoint orbit passing $\Ad^*_{b_0(x)}i^* \alpha\in (\mathrm{ker} \rho_y)^*$, which is a symplectic leaf in the Lie-Poisson manifold $(\mathrm{ker} \rho_y)^*$.  This proves that the fiber of $\lambda^*$ over $\beta\in T_y^*\mathcal{L}_x$ is 
the coadjoint orbit passing $\Ad_{ b_0(x)}^* i^* \alpha$ in $(\mathrm{ker} \rho_y)^*$. 

Observe the isomorphism \[ r^{-1}(x)\times_{G_x} \mathrm{ker} \rho_x\mapsto \mathrm{ker} \rho_y,\qquad [g,u]\mapsto \Ad_g u,\qquad g:x\to y.\]
We have
 \[A^*|_{\mathcal{L}_x}\cong T^* \mathcal{L}_x\oplus (r^{-1}(x)\times_{G_x} (\mathrm{ker} \rho_x)^*)\qquad \mathcal{O}_\alpha\cong T^*\mathcal{L}_x\times_{\mathcal{L}_x} (r^{-1}(x)\times_{G_x} \mathcal{S}_{i^*(\alpha)}),\]
 as fiber bundles over $\mathcal{L}_x$. Notice that $T^*\mathcal{L}_x \times_{\mathcal{L}_x} (r^{-1}(x)\times_{G_x} \mathcal{S}_{i^*\alpha})$ equals to 
 $r^{-1}(x)^\sharp \times_{G_x} \mathcal{S}_{i^*\alpha}$ as manifolds. The symplectic leaf $\mathcal{O}_\alpha$ is also the associated fiber bundle over $T^*\mathcal{L}_x$.
\end{proof}


From the above proof , we can also see
the symplectic leaf $\mathcal{O}_\alpha$ passing $\alpha\in A_x^*$ is isomorphic to $T^*\mathcal{L}_x\oplus (r^{-1}(x)\times_{G_x} \mathcal{S}_{i^*\alpha})$ as a fiber bundle over $\mathcal{L}_x$.

Now we shall use the splitting to put a symplectic structure on $T^*\mathcal{L}_x\times_{\mathcal{L}_x} (r^{-1}(x)\times_{G_x} \mathcal{S}_{i^*\alpha})=r^{-1}(x)^\sharp\times_{G_x} \mathcal{S}_{i^*\alpha}$  (as manifolds) to make the isomorphism symplectic.

If a splitting $\lambda: T\mathcal{L}_x\to A|_{\mathcal{L}_x}$ of the sequence 
\[0\to \ker \rho|_{\mathcal{L}_x}\xrightarrow{i} A|_{\mathcal{L}_x}\xrightarrow{\rho}  T\mathcal{L}_x\to 0\]
is chosen, we have a $\mathrm{ker} \rho|_{\mathcal{L}_x}$-valued $2$-form $R_\lambda\in \Omega^2(\mathcal{L}_x,\mathrm{ker} \rho|_{\mathcal{L}_x})$ defined by
\[R_\lambda(X,Y):=\lambda([X,Y]_S)-[\lambda(X),\lambda(Y)]_{A|_{\mathcal{L}_x}},\qquad X,Y\in \mathfrak{X}(\mathcal{L}_x).\]

\begin{Lem}
The decomposition \[T\mathcal{L}_x\oplus \mathrm{ker} \rho|_{\mathcal{L}_x}\cong A|_{\mathcal{L}_x},\qquad (X,U)\mapsto (\lambda(X),i(U))\] induces a Lie algebroid structure on $T\mathcal{L}_x\oplus \mathrm{ker} \rho|_{\mathcal{L}_x}$ with  the bracket given by 
\begin{eqnarray*}\label{4-term}
 [X,Y]=[X,Y]_S-R_\lambda(X,Y),\quad  [U,V]=[U,V]_{\mathrm{ker} \rho|_{\mathcal{L}_x}}, \quad  [X,U]=[\lambda(X), U]_{A|_{\mathcal{L}_x}}, 
\end{eqnarray*}
and the anchor being the projection to $T\mathcal{L}_x$.
\end{Lem}

Let us first look at two special cases.

\begin{Pro}
When $\alpha\in A_x^*$ such that $i^*\alpha=0$, the symplectic leaf $\mathcal{O}_{\alpha}$  is symplectically diffeomorphic to $T^*\mathcal{L}_x$  with the canonical symplectic structure. 
\end{Pro}
\begin{Pro}
If the isotropy Lie algebra bundle $\mathrm{ker} \rho |_{\mathcal{L}_x} \subset A$ is abelian, then the symplectic leaf $\mathcal{O}_\alpha$ passing $\alpha$ is symplectically diffeomorphic to $T^*\mathcal{L}_x$ with the symplectic structure defined at a point $\beta\in T_y^* \mathcal{L}_x$ by
\[\omega(\beta)=\omega_{can}(\beta)-p^*\langle R_\lambda,\Ad^*_{ b_0(x)}i^* \alpha\rangle,\qquad  \mathrm{any}\quad b_0\in \mathrm{LBis}_x\quad s.t. \quad \phi_{b_0}(x)=y.\]
Here $p:T^*\mathcal{L}_x\to \mathcal{L}_x$ is the projection and $\omega_{can}$ is the canonical symplectic structure on $T^*\mathcal{L}_x$.
\end{Pro}
\begin{proof}
Since the isotropy Lie group is abelian, the coadjoint orbit in $(\mathrm{ker} \rho_y)^*$ degenerates to a point, so we have the symplectic leaf through $\alpha$ is diffeomorphic to $T^*\mathcal{L}_x$. Explicitly, $\mathcal{O}_\alpha$ at a point $y$
is $T_y^*\mathcal{L}_x\times \{\Ad_{b_0(x)}^*i^* \alpha\}\cong T_y^*\mathcal{L}_x$, where $b_0\in \mathrm{LBis_x}$ is any bisection such that $\phi_{b_0}(x)=y$. It is easy to see $\Ad_{b_0(x)}^* i^*\alpha$ does not depend on the choice of $b_0$ since $\mathrm{ker} \rho_y$ is abelian.

For $X,Y\in \mathfrak{X}(\mathcal{L}_x)$, denote by $l_X,l_Y\in C^\infty (T^*\mathcal{L}_x)$ the linear functions on $T^*\mathcal{L}_x$ and $\tilde{X}_{l_X},\tilde{X}_{l_Y}$ the Hamiltonian vector fields on $T^*\mathcal{L}_x$ with the new Poisson structure. Then the symplectic $2$-form $\omega$ at $\beta\in T_y^* \mathcal{L}_x$ is 
\begin{eqnarray*}
\omega(\tilde{X}_{l_X},\tilde{X}_{l_Y})(\beta)&=&\{l_X,l_Y\}(\beta)\\ &=&[X,Y]_S(\beta)-\langle R_\lambda(X,Y),\Ad_{b_0(x)}^*i^* \alpha \rangle(y)\\ &=&\omega_{can}(\tilde{X}_{l_X},\tilde{X}_{l_Y})(\beta)-(p^*\langle R_\lambda,\Ad_{b_0(x)}^*i^* \alpha\rangle)(\tilde{X}_{l_X},\tilde{X}_{l_Y})(\beta).\end{eqnarray*}
Here we have used the fact $dp(\tilde{X}_{l_X})=dp(X_{l_X})=X$.  
Actually, by definition,  we have $\tilde{X}_{l_X}-X_{l_X}$ is $C^\infty(\mathcal{L}_x)$-linear, and then belongs to $\End(T^*\mathcal{L}_x)$. This implies that it is a vertical vector field. 
This justifies the first identity.
For $f,g\in C^\infty(\mathcal{L}_x)$, noticing $dp(\tilde{X}_{p^*f})=dp(X_{p^*g})=0$, we have
\[\omega(\tilde{X}_{p^*f},\tilde{X}_{p^*g})=0=(\omega_{can}-p^*\langle R_\lambda,i^*\alpha \rangle)(\tilde{X}_{p^*f},\tilde{X}_{p^*g}).\]
and
\[\omega(\tilde{X}_{l_X},\tilde{X}_{p^*f})=\omega_{can}(X_{l_X},X_{p^*f})=(\omega_{can}-p^*\langle R_\lambda,i^*(\alpha)\rangle)(\tilde{X}_{l_X},\tilde{X}_{p^*f}).\]
We complete the proof.
\end{proof}
It is seen that when the isotropy Lie algebra is abelian, the symplectic leaf is $T^*\mathcal{L}_x$, but the symplectic structure is the canonical one plus a "magnetic" term.

In general, since a symplectic leaf $\mathcal{O}_\alpha$ passing $\alpha\in A_x^*$ only relates with $A|_{\mathcal{L}_x}$, which is isomorphic the gauge algebroid $Tr^{-1}(x)/G_x$,  the general case actually boils down to the gauge algebroid case. The following result follows from the statement in \cite{GS}.

Let $\omega_{\mathcal{S}}$ be the symplectic form on $\mathcal{S}_{i^*\alpha}$ and $q:r^{-1}(x)\times \mathcal{S}_{i^*\alpha}\to r^{-1}(x)\times_{G_x} \mathcal{S}_{i^*\alpha}$ the projection onto $G_x$-orbits. Define $\widetilde{\omega_{\mathcal{S}}}\in \Omega^2(r^{-1}(x)\times_{G_x} S_{i^*\alpha})$ as 
\[\widetilde{\omega_{\mathcal{S}}}_{[g,\mu]}(w_1,w_2)={\omega_{\mathcal{S}}}(u_1,u_2),\]
where $u_j\in T_\mu \mathcal{S}_{i^*\alpha}$ such that $dq(0,u_j)=V_{[g,\mu]}(w_j)$ for $j=1,2$. Here $V_{[g,\mu]}$ is a projection of $T_{[g,\mu]} (r^{-1}(x)\times_{G_x} \mathcal{S}_{i^*\alpha})$ to the subspace of vertical vectors.

\begin{Thm}\label{mainrevise}
The symplectic leaf passing $\alpha\in A_x^*$ is symplectically isomorphic to 
\[T^*\mathcal{L}_x\times_{\mathcal{L}_x} (r^{-1}(x)\times_{G_x} S_{i^*\alpha})=r^{-1}(x)^\sharp\times_{G_x} \mathcal{S}_{i^*\alpha}\] (as manifolds) with the symplectic structure at a point
\[\omega(\beta+[g,\mu])=\omega_{can}(\beta)-p^*\langle R_\lambda, \Ad_g^* i^*\alpha\rangle+\widetilde{\omega_{\mathcal{S}}}([g,\mu]),\quad \beta\in T^*_y \mathcal{L}_x,[g,\mu]\in r^{-1}(x)\times_{G_x} \mathcal{S}_{i^*\alpha}, g:x\to y.\]
Here $\omega_{can}$ is the canonical symplectic form on $T^*\mathcal{L}_x$ and $p$ is the projection from $T^*\mathcal{L}_x$ to $\mathcal{L}_x$.
\end{Thm}
It is shown in this theorem that the symplectic form on a symplectic leaf involves the canonical symplectic form on $T^*\mathcal{L}_x$, the symplectic form on a cadjoint orbit of the isotropy Lie group and the curvature. We also refer to \cite{MP} for a more explicit formula of the symplectic structure for the gauge algebroid case.

For the gauge algebroid $TP/G$ of a principal $G$-bundle $P$ over $M$,  a splitting is actually a connection of the principal bundle. When a connection is chosen, the fiber bundle we obtained in Theorem \ref{mainrevise}  is precisely the phase space constructed by Sternberg for a classical particle in a Yang-Mills field.  See \cite{GS} for more discussion. 

Explicitly, let $P^\sharp$ be the pull-back bundle of the principal $G$-bundle $P\to M$  along the projection $T^*M\to M$,
which is the principal $G$-bundle over $T^*M$. 

\begin{Cor}\label{associated bundle}
When a connection $\theta$ is chosen, for $\alpha\in T^*P/G$ such that $J(\alpha)=[p,\mu]$, we have that the symplectic leaf $\mathcal{O}_\alpha$  is diffeomorphic to the associated bundle $P^\sharp\times_G \mathcal{S}_\mu$. Moreover, we can use the connection  to put a symplectic structure on $P^\sharp\times_G \mathcal{S}_\mu$ to make the above isomorphism a symplectically isomorphism. 
\end{Cor}

\begin{Ex}
 If $P=M\times G$ is the trivial principal $G$-bundle, then we have the gauge algebroid $TP/G=TM\times \mathfrak{g}$ and $T^*P/G=T^*M\times \mathfrak{g}^*$. So the symplectic leaf passing $(x,\mu)\in M\times \mathfrak{g}^*$ is $T^*M\times \mathcal{S}_\mu$ with the symplectic structure $p_1^*\omega_{can}+p_2^* \omega_{\mathcal{S}}$, where $\mathcal{S}_\mu$ is the coadjoint orbit passing $\mu$ of the Lie group $G$ on $\mathfrak{g}^*$ and $p_i$ is the projection of $T^*M\times \mathcal{S}_\mu$ to the $i$-th component.
\end{Ex}



\subsection{Examples}

Let us look at the symplectic leaves of the Lie-Poisson structure on $A^*$ for some familiar Lie algebroids. 

Let $A=(T^*_\pi M,\pi^\sharp, [\cdot,\cdot]_{\pi})$ be the Lie algebroid associated to a Poisson manifold $(M,\pi)$. Its anchor is $\pi^\sharp: T^*_\pi M\to TM$ and  the Lie bracket is given by
\[[\gamma,\gamma']_{\pi}=\mathcal{L}_{\pi^\sharp(\gamma)} \gamma'-\mathcal{L}_{\pi^\sharp(\gamma')} \gamma-d\pi(\gamma,\gamma'),\qquad \gamma,\gamma'\in \Omega^1(M).\]
Observe that the isotropy Lie algebra $\mathrm{ker}\pi^\sharp|_x$ is abelian.
The tangent bundle $TM$ is equipped with  a Lie-Poisson structure as the dual of $T^*_\pi M$, which is known as the tangent lift of the Poisson structure on $M$; see \cite{GU} for details.  
\begin{Pro}
With the above notations, the symplectic leaf passing through $\alpha\in T_x M$ is $T\mathcal{L}_x$, where $\mathcal{L}_x$ is the symplectic leaf passing $x\in M$ in the Poisson manifold $M$. 
\end{Pro}
\begin{proof}
We claim that the characteristic distribution generated by  Hamiltonian vector fields on $TM$ is actually generated by the family of \[\{X^\vee,X^c;X\in \mathrm{Ham}(M)\}\]  the vertical lift and complete lift of Hamiltonian vector fields on $M$. 
In fact,  for $f\in C^\infty(M)$, denote by $X_f$ the Hamiltonian vector field on $M$. Then we claim that the Hamiltonian vector fields  of $l_{df}\in C^\infty(TM)$ of the $1$-form $df\in \Omega^1(M)$ and $p^*f\in C^\infty(TM)$ on $TM$ are respectively
\[X_{l_{df}}=d X_f=X_f^c,\qquad X_{p^*f}=X_f^\vee.\]
Actually, since the flow of $dX_f\in \mathfrak{X}(TM)$ is $d\phi_t^{X_f}$, the tangent map of the flow of $X_f$ on $M$, we get
\[dX_f(dg) (Y_x)=\frac{d}{dt}|_{t=0}\langle dg, d\phi_t^{X_f}(Y_x)\rangle=\frac{d}{dt}|_{t=0}\langle (d\phi_t^{X_f})^*(dg),Y_x\rangle=\langle \mathcal{L}_{X_f}(dg),Y_x\rangle=d\{f,g\}(Y_x),\]
and \[(dX_f)(p^*g)(Y_x)=\frac{d}{dt}|_{t=0}p^*g(d\phi_t^{X_f}(Y_x))=\frac{d}{dt}|_{t=0}g(\phi_t^{X_f}(x))\circ p=p^*\{f,g\}.\]
Thus we get
\begin{eqnarray*}
X_{l_{df}}(dg)=\{df,dg\}=d\{f,g\}= dX_f (dg),\qquad X_{l_{df}}(p^*g)=p^*\{f,g\}=(dX_f)(p^*g).
\end{eqnarray*}
Namely, $X_{l_{df}}=dX_f$. To see $X_{p^*f}=X_f^\vee$, we have 
\[\langle X_{p^*f} (Y_x),dg\rangle=\langle p^*\{f,g\}, Y_x\rangle=\{f,g\}(x),\]
and 
\[\langle X^\vee_f (Y_x),dg\rangle=\frac{d}{dt}|_{t=0}\langle Y_x+tX_f(x),dg\rangle=\{f,g\}(x).\]
It is clear that both of $X_{p^*f}$ and $X_f^\vee$ are zero on $p^*g$, so they are equal. 
Note that vector fields on $T\mathcal{L}_x$ are generated by the vertical and complete lift of vector fields on $\mathcal{L}_x$. We get the result.
\end{proof}

Consider the Lie-Poisson manifold $\mathfrak{g}^*$ on the dual of  a Lie algebra $\mathfrak{g}$. We have the Lie algebroid structure  on $T^*\mathfrak{g}^*=\mathfrak{g}\triangleright \mathfrak{g}^*\to \mathfrak{g}^*$, which is the transformation algebroid relative to the coadjoint action of $\mathfrak{g}$ on $\mathfrak{g}^*$. 
Thus we get the tangent Poisson structure on $T\mathfrak{g}^*$. We refer to \cite[Example 1]{GU} for  the symplectic leaves of $T\mathfrak{g}^*$ when $\mathfrak{g}=\mathfrak{so}_3(\mathbbm{R})$.
\begin{Ex}
Let $I$ be a Lie algebra bundle over $M$. Then the symplectic leaf passing $\alpha\in I_x^*$ is $\mathcal{S}_\alpha$, the coadjoint orbit in the dual of the Lie algebra $I_x$.
\end{Ex}

\begin{Ex}
Let $A=F$ be a regular integrable distribution of a manifold $M$. By Theorem \ref{main2}, the symplectic leaf passing one point $\alpha\in F^*_x$ is \[F^*|_{\mathcal{L}_x}=T^*M/(T\mathcal{L}_x)^\perp\cong T^*\mathcal{L}_x,\]where $\mathcal{L}_x$ is the leaf passing $x$ of the distribution.  If $M $ is a product manifold $N\times P$ and $A=\pi^*(TN)$, where $\pi$ is the projection from $M$ to $N$, then the symplectic leaf passing through $(n,p)\in M$ is $T^*N \times \{p\}$.

For the same reason, if the anchor of a Lie algebroid $A$ is injective, then the symplectic leaf passing $\alpha\in A^*_x$ is $A^*|_{\mathcal{L}_x}=\rho^*(T^*\mathcal{L}_x)$.
\end{Ex}
\begin{Ex}
Let $A$ be the transformation Lie algebroid $\mathfrak{g} \triangleright M$ with the Lie algebra $\mathfrak{g}$ left acting on $M$. For $x\in M$, a local bisection is characterized by a smooth function $F: x\in U\to G$ such that the induced map $f: U\to f(U); f(x)=F(x)x$ is a diffeomorphism. 

For $(u,x)\in \mathfrak{g}\times M$ and a local bisection $b(x)=(F(x),x)$, we have
\[\Ad_{j_x^1 b} (u,x)=\Ad_{F(x)}u+dR_{F(x)^{-1}} dF(\hat{u}),\]
where $\hat{u}\in \mathfrak{X}(M)$ is the fundamental vector field of $u$. By definition, we have $b^{-1}(x)=(F(f^{-1}(x))^{-1},x)$. We write $\tilde{F}:f(U)\to G$ as $\tilde{F}(x)=F(f^{-1}(x))^{-1}$.

For the Lie-Poisson manifold $A^*=\mathfrak{g}^*\times M$, by (\ref{sym leaf}), the symplectic leaf  passing a point $(\mu,x)\in A^*$ is \[\{F(x)x,\Ad_{F(x)^{-1}}^*\mu+\rho^*(d\tilde{F})^*(dR_{F(x)})^*\mu+\rho^*(df^{-1})^*\gamma), \forall \gamma\in T_{x}^* M, F\in \mathrm{LBis_x}(M\times G)\},\]
where $\rho: \mathfrak{g}\times M\to TM$ is the anchor map given by the action.

In particular, the symplectic leaf passing $(0,x)$ is 
\begin{eqnarray*}
\mathcal{O}_{(0,x)}&=&\{(F(x)x,\rho^*(df^{-1})^* \gamma), \forall \gamma\in T_x^* M, F\in \mathrm{LBis_x}(M\times G)\}=\mathrm{Im} \rho^*|_{\mathcal{L}_x}=\rho^*(T^*\mathcal{L}_x),
\end{eqnarray*}
where  $\mathcal{L}_x$ is the action orbit of $G$ on $M$ through $x$.



If the $G$-action on $M$ is trivial, then the symplectic leaf is $\mathcal{S}_\mu\times \{x\}$, where $\mathcal{S}_\mu$ is the coadjoint orbit of $G$ passing $\mu\in \mathfrak{g}^*$ with the known symplectic structure. If this action is free, then the symplectic leaf is $ \mathfrak{g^*}\times \mathcal{L}_x$, which is symplectically diffeomorphic to $T^*\mathcal{L}_x$ with the canonical symplectic form by the map $\rho^*:T^*\mathcal{L}_x\to \mathfrak{g}^*\times \mathcal{L}_x$.

\end{Ex}

\section{The symplectic groupoid $T^*\mathcal{G}$ over $A^*$ and symplectic leaves}
In this section, we calculate the groupoid orbits of the symplectic groupoid $T^*\mathcal{G}\rightrightarrows A^*$, which turns out to be exactly the affine coadjoint orbits of $\mathcal{J}\mathcal{G}\ltimes T^*M$ on $A^*$ and thus the symplectic leaves on $A$*. This recovers a general result of symplectic groupoids for this special case \cite{CDW} .

For a symplectic groupoid with connected fibers, the base manifold comes equipped with a Poisson structure.  The groupoid orbits, i.e., the singular foliation where you identify points which are the source and target of the same element of the Lie groupoid,  and the symplectic leaves of the Poisson structure on the base manifold coincide \cite{MW, Mackenzie}.


For a Lie algebroid $A$ with Lie groupoid $\mathcal{G}$, the cotangent bundle $T^*\mathcal{G}$ is a symplectic Lie groupoid over $A^*$ and the induced Poisson structure on $A^*$ is exactly the Lie-Poisson structure.  

Let us first recall the groupoid structure on $T^*\mathcal{G}$. As $T_x\mathcal{G}=T_xM \oplus A_x$, any $\alpha\in A^*_x$ corresponds to an element $i(\alpha)\in T^*_x \mathcal{G}$ determined by
\[\langle i(\alpha), Z+u\rangle=\langle \alpha, u\rangle,\qquad \forall Z\in T_xM, u\in A_x.\]
The source of an element $\xi\in T^*_g\mathcal{G}$ with $g\in \mathcal{G}$ is
\[\langle r(\xi), u\rangle=\langle \xi, d L_g(u-\rho(u))\rangle,\quad \quad \forall u\in A_{r(g)}.\]
Note that here we treat $A=\mathrm{ker} dr_{\mathcal{G}}|_M$ and 
$u\mapsto -\mathrm{inv}(u)=u-\rho(u)$ is a map $\mathrm{ker} dr_{\mathcal{G}}|_x\to \mathrm{ker} dl_{\mathcal{G}}|_x$. This makes sense of the left translation.
The target is 
\[\langle l(\xi),v\rangle=\langle \xi, d R_g(v)\rangle,\quad \quad \forall v\in A_{l(g)}.\] For a multiplicable pair $(g,h)\in \mathcal{G}^{(2)}$, if $\xi\in T_g^*\mathcal{G}$ and $\eta\in T_h^*\mathcal{G}$ are multiplicable, then the product is the element $\xi\times_{T^*\mathcal{G}} \eta\in T^*_{gh}\mathcal{G}$ such that
\begin{eqnarray}\label{multi}
(\xi\times_{T^*\mathcal{G}} \eta)(Tm(X,Y))=\xi(X)+\eta(Y),
\end{eqnarray}
for $X\in T_g \mathcal{G}$ and $Y\in T_h \mathcal{G}$ such that $dr_{\mathcal{G}}(X)=dl_{\mathcal{G}}(Y)$.

\begin{Rm}
The source and target maps of the Lie groupoid $T^*\mathcal{G}\rightrightarrows A^*$ are actually the dual of the following two inclusions
\[A\hookrightarrow T\mathcal{G},\quad u\to \overrightarrow{u};\quad \quad A\hookrightarrow T\mathcal{G}, \quad u\to \overleftarrow{u},\] 
of which $\Gamma(A)$ is seen as spaces of left and right vector fields on $\mathcal{G}$ respectively. 

To show the multiplication (\ref{multi}) is well-defined, we first note that $Tm:T\mathcal{G}\times_{TM} T\mathcal{G}\to T\mathcal{G}$ is surjective. It suffices to check that if $Tm(X,Y)=0$, then $\xi(X)+\eta(Y)=0$ for a multiplicable pair $(\xi,\eta)$. Actually, the condition $Tm(X,Y)=0$ for $X\in T_g G$ and $Y\in T_h Y$ implies that $X=-d L_g \mathrm{inv} (u)$ and $Y=d R_{h} (u)$ for some $u\in A_{r(g)}$. It is direct to see that $\xi(X)+\eta(Y)=0$ is equivalent to the multiplicable condition $r(\xi)=l(\eta)$.
\end{Rm}
\begin{Ex}\label{symp groupoid}
We have $T^*G\cong G\triangleright_{\Ad^*} \mathfrak{g}^*\rightrightarrows \mathfrak{g}^*$, which is the transformation groupoid with respect to the coadjoint action of $G$ on $\mathfrak{g}^*$; and \[T^*(M\times M)=T^*M\times T^*M\rightrightarrows T^*M,\]
where the structure maps are given by
\[r(\gamma,\gamma')=-\gamma', \qquad l(\gamma,\gamma')=\gamma,\qquad i(\gamma)=(\gamma,-\gamma),\qquad (\gamma,\gamma')(-\gamma',\gamma'')=(\gamma,\gamma'').\]
For the transformation groupoid $G\triangleright M$, we have
\[T^*(G\triangleright M)=T^*G\times T^*M\rightrightarrows  \mathfrak{g}^*\times M,\]
whose groupoid structures are as follows. The source and target maps are
\[l(\alpha,\gamma)=(R_h^*\alpha,hx),\qquad r(\alpha,\gamma)=(L_h^*\alpha-\rho^*(\gamma),x),\qquad \alpha\in T_h^* G,\gamma\in T_x^* M,\]
where $\rho^*: T^*M\to \mathfrak{g}^*\times M$ is the dual map of the action (anchor).
And for a multiplicable pair $(\beta,\gamma')\in T_{g}^*G\times T_{hx}^* M$ and $(\alpha,\gamma)\in T_h^* G\times T_x^*M$, their product is 
\[(\beta,\gamma')\cdot (\alpha,\gamma)=(R_{h^{-1}}^* \beta,L_h^* \gamma'+\gamma)\in T_{gh}^*G\times T_x^*M.\]
\end{Ex}


The groupoid orbits of $T^*\mathcal{G}$ on $A^*$ are also the orbits of the action of  $T^*\mathcal{G}$  on $A^*$ by $\xi\cdot \alpha=l(\xi)$ when $\alpha=r(\xi)$. The symplectic groupoid looks quite different from the first jet groupoid of a Lie groupoid from  Example \ref{jet groupoid} and \ref{symp groupoid}, but we shall prove that they have the same orbits on $A^*$, which are the symplectic leaves.

\begin{Thm}\label{symplectic leaf}
The groupoid orbits of the symplectic groupoid $T^*\mathcal{G}$ on $A^*$
coincide with the orbits of the affine coadjoint action of $\mathcal{J}\mathcal{G}\ltimes T^*M$ on $A^*$.
\end{Thm}
\begin{proof}
Two points  $\alpha\in A^*_y$ and $\beta\in A^*_x$ are in the same groupoid orbit iff there is an element $g:x\to y$ in $\mathcal{G}$ and a covector $\xi\in T_g^*\mathcal{G}$ such that $r(\xi)=\beta$ and $l(\xi)=\alpha$.


If $g: x\to y$ and $\alpha\in A_y^*$, let $b:U\subset M\to \mathcal{G}$ be a local bisection of $\mathcal{G}$ such that  $b(x)=g$. Following from
\[\langle \alpha, v\rangle=\langle l(\xi),v\rangle=\langle \xi, R_{g} (v)\rangle=\langle R_b^* \xi, v\rangle,\quad \quad \forall v\in A_y,\xi\in T_g^*\mathcal{G},\]
we have 
\[\alpha-R_b^*\xi=\gamma\]
for some $\gamma\in T_y^*M$, which implies that
 \[\xi=R_{b^{-1}}^* \alpha-R_{b^{-1}}^* \gamma.\]
Here $R_b$ is the isomorphism given by the right translation 
\[R_{b}: A_y\oplus T_y M\to T_g \mathcal{G},\qquad R_b (w)=\frac{d}{dt}|_{t=0} \phi_t^w(y) b(\phi_b^{-1}(r(\phi_t^w(y))),\]
where $\phi_t^w$ is a flow of $w$. It is seen that the inverse of $R_b$ is $R_{b^{-1}}$ and it satisfies that $R_b|_{A_y}=dR_{g}$.  For similar reason, we get 
\[L_{g}(u-\rho(u))=L_b u-L_b \rho(u),\qquad \forall u\in A_x.\]
Then we get
\[\langle \beta, u\rangle=\langle r(\xi), u\rangle=\langle R_{b^{-1}}^*\alpha-R_{b^{-1}}^*\gamma, L_b u-L_b\rho(u)\rangle=
\langle \Ad_{b^{-1}}^* \alpha,u \rangle+\langle \rho^*(\Ad_{b^{-1}}^* \gamma), u\rangle.\]
Here we have used the fact that $\Ad_b$ preserves $A$ and $TM$. This implies that
\[\beta=\Ad_{b^{-1}}^* \alpha+\rho^*(\Ad_{b^{-1}}^* \gamma)\]
for some $\gamma\in T^*_y M$.
Note that $\Ad_b^*$ only depends on the equivalent class $j_x^1 b$ of $b$ in $\mathcal{J}\mathcal{G}$ and $\Ad_{j_{\phi_b(x)}^1 b^{-1}}^* \gamma=(d\phi_b)^*\gamma$. By Theorem \ref{main2}, we finish the proof.
\end{proof}
The jet groupoid $\mathcal{J}\mathcal{G}$ coadjoint acts on $A^*$. It also left acts on  $T^*\mathcal{G}$.
A bisection $b$ of a Lie groupoid $\mathcal{G}$ left acts on $\mathcal{G}$ by $b\cdot g=b(l(g))g$. Its tangent map defines an action of $\mathcal{J}\mathcal{G}$ on $T\mathcal{G}$ with moment map $l\circ p:T\mathcal{G}\to M$:
\[L_{j_x^1 b}: T_g \mathcal{G}\to T_{b(x)g} \mathcal{G},\qquad L_{j_x^1 b}(X):=dL_{b}(X),\qquad X\in T_g G, l(g)=x,\]
where $p:T\mathcal{G}\to \mathcal{G}$ is the projection. Taking the dual, we get an action of $\mathcal{J}\mathcal{G}$ on $T^*\mathcal{G}$ with moment map $l\circ p:T^*\mathcal{G}\to M$:
\[L_{j_x^1 b}^*:T_g^* \mathcal{G}\to T_{b(x)g}^* \mathcal{G},\qquad \langle L_{j_x^1 b}^*(\xi),X\rangle:=\langle \xi, L_{j_{\phi_b(x)}^1 b^{-1}}(X)\rangle,\qquad X\in T_{b(x)g} \mathcal{G}, l(g)=x. \]
In the following lemma, we discuss the fundamental vector fields of this action.
\begin{Lem}
The fundamental vector field of $u\in \Gamma(A)$ and $D\in \Gamma(\Hom(TM,A))$ on $T^*\mathcal{G}$ are   \[\hat{u}=[\overrightarrow{u},\cdot]_{T\mathcal{G}},\qquad \hat{D}=-R\circ D\circ dl,\]
where $R:\Gamma(A)\to \mathfrak{X}(\mathcal{G})$ is the right translation and $r:\mathcal{G}\to M$ is the source map.
\end{Lem}
\begin{Pro}
For any  $\mathbbm{d}u+D\in \Gamma(\mathcal{J}A)$, we have its fundamental vector field on $T^*\mathcal{G}$ is a Poisson vector field iff $D$ satisfies
\[RD[dl(X),dl(Y)]=[RDdl(X),Y]+[X,RDdl(Y)],\qquad X,Y\in \mathfrak{X}(\mathcal{G}).\]
It is a Hamiltonian vector field iff $D=0$ and the Hamiltonian function is $l_{\overrightarrow{u}}\in C^\infty(T^*\mathcal{G})$.
\end{Pro}

The symplectic groupoid $T^*\mathcal{G}\rightrightarrows A^*$ left acts on itself with moment map $J=l:T^*\mathcal{G}\to A^*$. We find this moment map is equivariant with respect to the $\mathcal{J}\mathcal{G}$-actions.
\begin{Pro}
The moment map $J: T^*\mathcal{G}\to A^*$ is $\mathcal{J}\mathcal{G}$-equivariant, i.e.,$J\circ L_s^*=\Ad_s^*\circ J$ for $s\in \mathcal{J}\mathcal{G}$. Namely, the diagram
\begin{equation*}
		\begin{tikzcd}
			T^*\mathcal{G} \arrow{d}{J} \ar{r}{L_s^*} &T^*\mathcal{G} \arrow{d}{J}\\
			A^* \arrow{r}{\Ad_s^*} &A^*
		\end{tikzcd}
\end{equation*}
is commutative.
\end{Pro}
\begin{proof}
Explicitly, we shall check
\begin{eqnarray}\label{eq}
J(L_{j_x^1 b}^* \xi)=\Ad_{j_x^1 b}^* J(\xi),\qquad  \xi\in T_g^*\mathcal{G}, l(g)=x.
\end{eqnarray}
Let $\phi_b(x)=y$. For $u\in A_{y}$, we have 
\[\langle \Ad_{j_x^1 b}^* J(\xi),u\rangle=\langle l(\xi),\Ad_{j_y^1 b^{-1}} u\rangle=\langle \xi, R_{g} \Ad_{j_y^1 b^{-1}} u\rangle\]
and 
\[\langle J(L_{j_x^1 b}^* \xi), u\rangle=\langle L_{j_x^1 b}^* \xi, R_{b(x)g} u\rangle=\langle \xi, L_{j_y^1 b^{-1}} R_{b(x)g} u\rangle.\]
Then the equality (\ref{eq}) holds since we have
\[R_g \mathrm{AD}_{b^{-1}} (h)=b^{-1}(l(h))h (b^{-1}(y))^{-1}g=b^{-1}(l(h))h b(x)g=(L_{ b^{-1}}\circ R_{b(x)g})(h),\]
for any $h\in r^{-1}(y)$. 
\end{proof}

\begin{Rm} It is direct to check $r(L_{j_x^1 b}^* \xi)=r(\xi)$ for $\xi\in T_g^*\mathcal{G}$. Then
Equation (\ref{eq}) explains the explicit corresponding relation between the coadjoint orbit passing $l(\xi)\in A^*$ of $\mathcal{J}\mathcal{G}$  and the groupoid orbit of $T^*\mathcal{G}$ passing $r(\xi)\in A^*$. Besides this, we can further consider the affine coadjoint action and get a relation
\[J(L_{j_x^1 b}^* \xi+\rho^*(d \phi_{b}^{-1})^*\gamma))=\Ad_{j_x^1 b}^* J(\xi)
+\rho^*(d\phi_b^{-1})^*\gamma,\qquad \gamma\in T_x^* M,\] 
of which the right hand side is in the affine coadjoint orbit of $l(\xi)\in A_{x}^*$.
\end{Rm}

We can also identify a symplectic leaf with a symplectic reduced space of the symplectic groupoid $T^*\mathcal{G}$. See \cite{MW} for the reduction theorem of a symplectic groupoid acting on a symplectic manifold.


With a Lie groupoid $\Gamma$ acting on itself from the right ($J=r$), the quotient space $J^{-1}(\alpha)/\Gamma_\alpha$ is naturally isomorphic to the orbit of $\alpha$ in the base manifold $M$ by the map $g\to l(g)$. Here $\Gamma_\alpha=l^{-1}(\alpha)\cap r^{-1}(\alpha)$ is the isotropy group at $\alpha\in M$. 


The symplectic groupoid $T^*\mathcal{G}$ symplectically acts on itself with the moment map $r: T^*\mathcal{G}\to A^*$. Then we get 
\begin{Pro}
For an $r$-connected Lie groupoid $\mathcal{G}$ with Lie algebroid $A$, the symplectic leaves of $A^*$ are exactly the reduced manifolds for the action of $T^*\mathcal{G}$ on itself.
\end{Pro}

\end{document}